\documentclass{imsart}
\usepackage[utf8]{inputenc}
\RequirePackage[numbers]{natbib}
\usepackage{amscd,amsmath,amssymb,amsfonts,pdfsync,xargs,stmaryrd}
\usepackage{ntheorem}
\usepackage{latexsym,bbm,ifthen,twoopt}
\usepackage[shortlabels]{enumitem}
\usepackage{graphicx}
\usepackage{color,framed,xargs}
\usepackage[textwidth=4cm]{todonotes}
\usepackage{hyperref}
\usepackage{mathtools}

\usepackage{aliascnt}
\usepackage{cleveref}

\newtheorem{theorem}{Theorem}

\newaliascnt{lemma}{theorem}
\newtheorem{lemma}[lemma]{Lemma}
\aliascntresetthe{lemma}
\crefname{lemma}{lemma}{lemmas}
\Crefname{Lemma}{Lemma}{Lemmas}

\newaliascnt{proposition}{theorem}
\newtheorem{proposition}[proposition]{Proposition}
\aliascntresetthe{proposition}
\crefname{proposition}{proposition}{propositions}
\Crefname{Proposition}{Proposition}{Propositions}

\newaliascnt{corollary}{theorem}

\aliascntresetthe{corollary}
\crefname{corollary}{corollary}{corollaries}
\Crefname{Corollary}{Corollary}{Corollaries}

\theoremstyle{definition}
\newtheorem{definition}{Definition}
\theoremstyle{exercise}

\newaliascnt{remark}{definition}
\newtheorem{remark}[remark]{Remark}
\aliascntresetthe{remark}
\crefname{remark}{remark}{remarks}
\Crefname{remark}{Remark}{Remarks}

\newaliascnt{example}{theorem}

\aliascntresetthe{example}
\crefname{example}{example}{examples}
\Crefname{example}{Example}{Examples}

\newenvironment{proof}{\textit{Proof.} }{\hfill $\square$\\}

\newenvironment{hyp}[1]{
\refstepcounter{hyp#1}
\begin{itemize}[label=({\bf #1\arabic{hyp#1}}),resume=hyp#1]\begin{sf}}
{\end{sf}\end{itemize}}

\newcommandx{\refhyp}[3][1=,3=]{
\ifthenelse{\equal{#1}{}}
{\emph{\ref{#2}}}
{
\ifthenelse{\equal{#3}{}}
{\hspace{-2.5mm}{(\textbf{#1\ref{#2}})}\hspace{-2.5mm}}
{\hspace{-2.5mm}(\textbf{#1\ref{#2}--\ref{#3}})\hspace{-2.5mm}}
}
}

\newcommand{\as}{\mbox{-}\mathrm{a.s.}}
\newcommand{\pp}{\mbox{-}\mathrm{a.e.}}

\newcommandx\A[2][1=]{
\ifthenelse{\equal{#1}{}}
{\hspace{-1mm}(\textbf{A\ref{#2}})\hspace{-1mm}}
{\hspace{12mm}(\textbf{A\ref{#1}--\ref{#2}})\hspace{-1mm}}
}

\newcommand{\borel}[1]{\mathcal{B}(#1)}

\newcommandx\C[2][1=]{
\ifthenelse{\equal{#1}{}}
{\hspace{-1mm}(\textbf{C\ref{#2}})\hspace{-1mm}}
{\hspace{-1mm}(\textbf{C\ref{#1}--\ref{#2}})\hspace{-1mm}}
}

\newcommand{\chunk}[3]{{#1}_{#2:#3}}
\newcommandx{\cond}[2]{
\lrb{\left. #1\right| #2}
}
\newcommandx\dens[3][1=,3=]%
{
\ifthenelse{\equal{#1}{}}
{\ifthenelse{\equal{#3}{}}{p(#2)}{p_{#3}(#2)}}
{\ifthenelse{\equal{#3}{}}{p_{#1}(#2)}{p_{#1}^{#3}(#2)}}
}

\newcommand{\D}{\mathsf{\Delta}}
\newcommand{\Dhmm}{\overline{\mathsf{\Delta}}}
\newcommand{\e}{\mathrm{e}}
\newcommandx{\ens}[3][1=]
{
\ifthenelse{\equal{#1}{}}
{\left\{#2:\eqsp #3\right\}}
{\{#2:\eqsp #3\}}
}

\newcommand{\eqsp}{}
\newcommand{\eqspeq}{\quad \quad} 
\newcommand{\eqdef}{\coloneqq}

\newcommandx\esp[3][1=,3=]{{\mathbb E}_{#1}^{#3} \left[#2\right]}
\newcommandx\cesp[4][1=,4=]{{\mathbb E}_{#1}^{#4} \left[\left. #2\right| #3\right]}

\newcommandx\espTxt[3][1=,3=]{
\ifthenelse
{\equal{#3}{}}
{{\mathbb E}_{#1} (#2)}
{{\mathbb E}_{#1} (#2|#3)}
}

\newcommand{\ensemble}[2]{\left\{#1\,:\eqsp #2\right\}}

\newcommandx\F[2][1=]{
\ifthenelse{\equal{#1}{}}
{\hspace{-1mm}(\textbf{F\ref{#2}})\hspace{-1mm}}
{\hspace{-1mm}(\textbf{F\ref{#1}--\ref{#2}})\hspace{-1mm}}
}

\newcommand{\fdPOMM}{fdPOMM}
\newcommandx\feyn[2][1=]
{
\ifthenelse{\equal{#1}{}}
{\ensuremath{\Phi \langle #2 \rangle}}
{\ensuremath{\Phi \langle #2 \rangle (#1)}}
}

\newcommand{\filtletter}{\bar{\phi}}
\newcommandx\filt[3][1=,3=]
{
\ifthenelse
{\equal{#1}{}}
{\ensuremath{{\filtletter}\langle #2 \rangle }}
{\ensuremath{{\filtletter}_{#1} \langle #2 \rangle}}%
}

\newcommand{\gd}[2]{p_{#1,#2}}
\newcommand{\gdstar}[1]{{p^*_{#1}}}

\newcommand{\kernel}{\mathbf{K}}

\newcommandx{\kullback}[3][1=]{
\ifthenelse{\equal{#1}{}}
{\mathrm{KL}\left(#2\|#3\right)}
{\mathrm{KL}(#2\|#3)}
}

\renewcommand{\ker}[1]{\mathbf{#1}}

\newcommandx\LGSS[2][1=]{
\ifthenelse{\equal{#1}{}}
{\hspace{-1mm}(\textbf{LGSS\ref{#2}})\hspace{-1mm}}
{\hspace{-1mm}(\textbf{LGSS\ref{#1}--\ref{#2}})\hspace{-1mm}}
}
\newcommand{\lleb}{\lambda^{\mathrm{Leb}}}

\newcommandx{\Lkh}[2][1=]{
\ifthenelse{\equal{#1}{}}
{\boldsymbol{\operatorname{L}}\langle{#2}\rangle}
{\boldsymbol{\operatorname{L}}_{#1}\langle{#2}\rangle}
}

\newcommandx\likeli[3][1=]{\pi_{#1}\langle #2 \rangle(#3)}

\newcommand{\lr}[1]{\left( #1\right)}

\newcommand{\lrb}[1]{\left[ #1\right]}
\newcommand{\lrc}[1]{\left\{ #1\right\}}

\newcommand{\mcf}{\mathcal F}
\newcommand{\mcg}{\mathcal G}

\newcommand{\mct}{\mathcal T}
\newcommand{\mcu}{\mathcal U}

\newcommandx\NL[2][1=]{
\ifthenelse{\equal{#1}{}}
{\hspace{-1mm}(\textbf{NL\ref{#2}})\hspace{-1mm}}
{\hspace{-1mm}(\textbf{NL\ref{#1}--\ref{#2}})\hspace{-1mm}}
}

\newcommand{\nset}{\mathbb{N}}

\newcommand{\nsetpos}{\mathbb{N}^\ast}
\newcommand\1[1]{\mathbbm{1}_{#1}}
\newcommand\oneSub[1]{\mathbbm{1}_{#1}}

\newcommand\indiceTxt[1]{\mathbbm{1}\lrc{#1}}

\newcommand{\piv}{{\pi_{\star}}}
\newcommand{\pix}{\pi^{X}}

\newcommand{\py}[3]{p_{#1,#2}(#3)}
\newcommandx{\psup}[2][1=]{\hat p_{#1}(#2)}

\newcommandx\partfilt[3][1=,3=]
{
\ifthenelse
{\equal{#1}{}}
{\ensuremath{\filtletter^N \langle #2 \rangle }}
{\ensuremath{{\filtletter}_{#1}^N \langle #2 \rangle}}%
}

\newcommandx\partpred[3][1=,3=]%
{%
\ifthenelse
{\equal{#1}{}}
{\ensuremath{{\phi}^N\langle #2 \rangle }}
{\ensuremath{{\phi}^{N}_{#1} \langle #2 \rangle}}%
}

\newcommand{\Pblock}[2][]
{\ifthenelse{\equal{#1}{}}{\boldsymbol{\operatorname{L}}\langle#2\rangle}{\boldsymbol{\operatorname{L}}^{#1}\langle#2\rangle}
}
\newcommand{\pblock}[2][]
{\ifthenelse{\equal{#1}{}}{\mathbf{\ell}\langle#2\rangle}{\mathbf{\ell}^{#1}\langle #2\rangle}
}

\def\PE{\mathbb{E}}
\def\PP{\mathbb{P}}

\newcommandx\pred[3][1=,3=]%
{%
\ifthenelse
{\equal{#1}{}}
{\ensuremath{{\phi}\langle #2 \rangle }}
{\ensuremath{{\phi}_{#1} \langle #2 \rangle}}%
}

\newcommandx\prob[3][1=,3=]{
\ifthenelse
{\equal{#3}{}}
{{\mathbb P}_{#1}\left(#2\right)}
{{\mathbb P}_{#1} \left(\left. #2\right| #3\right)}
}

\newcommandx\probTxt[3][1=,3=]{
\ifthenelse
{\equal{#3}{}}
{{\mathbb P}_{#1} (#2)}
{{\mathbb P}_{#1} ( #2 | #3)}
}

\newcommand{\prior}{\lambda}
\newcommandx{\post}[2][1=]{
\lambda_{#1}\langle #2\rangle}

\newcommand{\rme}{\mathrm{e}}

\newcommand{\rset}{\mathbb{R}}
\newcommand{\rsetpos}{\mathbb{R}^\ast_+}
\def\rmd{\mathrm{d}}

\newcommandx\secfact[2][1=]{\Lambda_{#1} \langle #2 \rangle (h)}

\newcommandx\secfactub[2][1=]{\Upsilon_{#1} \langle #2 \rangle (h)}

\newcommandx\supnorm[2][1=]{\ensuremath{\left\|#2\right\|^{#1}_{\infty}}}

\newcommandx\sequence[3][2=n,3=\zset]{\ensuremath{(#1_{#2})_{#2 \in #3}}}
\newcommandx\sequencen[2][2=n\in\zset]{\ensuremath{(#1)_{#2}}}

\newcommand{\thetapr}{\theta'}
\newcommand{\tgdstar}[1]{\tilde{p}^*_{#1}}
\newcommand{\tgdn}[1]{\tilde{p}_{#1}}
\newcommand{\bgdn}[1]{\overline{p}_{#1}}
\newcommand{\bgdstar}[1]{\overline{p}^*_{#1}}

\newcommand{\tensprod}{\varotimes}
\newcommandx{\Tk}[1][1=]{\ker{Q}_{#1}}
\newcommandx{\Td}[1][1=]{q_{#1}}
\newcommand{\thv}{{\theta_\star}}
\newcommandx{\Tdx}[1][1=]{k_{#1}}
\newcommandx{\Tx}[1][1=]{\ker{K}_{#1}}
\newcommandx{\Tdy}[1][1=]{g_{#1}}

\newcommandx\variance[3][1=]{\sigma_\Xinit^{#1} \langle #2 \rangle (#3)}

\newcommand{\wrt}{with respect to}
\newcommand{\vnorm}[2]{\|#1\|_{#2}}

\DeclareMathOperator*{\basicwc}{\Longrightarrow}
\newcommand\weakconv[1]{\basicwc_{#1}}
\def\Xinit{\eta}
\def\Xinitv{{\eta_\star}}

\newcommand{\zset}{\mathbb{Z}}

\newcommand{\URoot}[1][]
{\ifthenelse{\equal{#1}{}}{\ensuremath{R}}{\ensuremath{R}_{#1}}}
\newcommand{\VRoot}[1][]
{\ifthenelse{\equal{#1}{}}{\ensuremath{S}}{\ensuremath{S}_{#1}}}
\newcommand{\UCov}[1][]%
{%
\ifthenelse{\equal{#1}{}}{\URoot {^t\URoot}}{\URoot[#1] {^t\URoot[#1]}}%
}
\newcommand{\VCov}[1][]%
{%
\ifthenelse{\equal{#1}{}}{\VRoot {^t\VRoot}}{\VRoot[#1] {^t\VRoot[#1]}}%
}

\def\Xset{\mathsf{X}}
\def\Yset{\mathsf{Y}}
\def\Zset{\mathsf{Z}}

\def\Xsigma{\mathcal{X}}
\def\Ysigma{\mathcal{Y}}
\def\Zsigma{\mathcal{Z}} 

\date{\today}
\begin{document}

\begin{frontmatter}
\title{Posterior consistency for partially observed Markov models}
\runtitle{Posterior consistency for fPOMMs}

\begin{aug}
 \author{\fnms{Randal}  \snm{Douc} \thanksref{rd} \ead[label=e1]{randal.douc@telecom-sudparis.eu}}
  \author{\fnms{Jimmy} \snm{Olsson} \thanksref{jo}\ead[label=e2]{jimmy@maths.lth.se}}
  \and \author{\fnms{Fran\c{c}ois} \snm{Roueff} \thanksref{fr}\ead[label=e3]{roueff@telecom-paristech.eu}}
\runauthor{R. Douc, J. Olsson and F. Roueff}

\thankstext{rd}{SAMOVAR, CNRS UMR 5157, Institut T\'el\'ecom/T\'el\'ecom SudParis,  9 rue Charles Fourier, 91000 Evry.}
\thankstext{jo}{KTH Royal Institute of Technology, Stockholm, Sweden. Jimmy Olsson is supported by the Swedish Research Council, Grant 2011-5577.}
\thankstext{fr}{LTCI, CNRS 5141, Institut T\'el\'ecom/T\'el\'ecom Paristech,  42 rue Barrault, 75000 Paris.}
\end{aug}

\begin{abstract}
In this work we establish the posterior consistency for a para\-metrized family of partially observed, fully dominated Markov models. As a main assumption, we suppose that the prior distribution assigns positive probability to all neighborhoods of the true parameter, for a distance induced by the expected Kullback-Leibler divergence between the family members' Markov transition densities. This assumption is easily checked in general. In addition, under some additional, mild assumptions we show that the posterior consistency is implied by the consistency of the maximum likelihood estimator. The latter has recently been established also for models with non-compact state space. The result is then extended to possibly non-compact parameter spaces and non-stationary observations. Finally, we check our assumptions on examples including the partially observed Gaussian linear model with correlated noise and a widely used stochastic volatility model.
\end{abstract}
\end{frontmatter}

\section{Introduction}


We consider a very general framework where a bivariate Markov chain
$\sequence{Z}[n][\nset]$ taking on values in some
product state space $\Zset = \Xset \times \Yset$, i.e., $Z_n = (X_n, Y_n)$, is only partially observed
through the second component $\sequence{Y}[n][\nset]$. In this model, which we
refer to as a \emph{partially observed Markov model} (POMM)
(\cite{pieczynski:2003} uses the alternative term \emph{pairwise Markov
  chain}), any statistical inference has to be carried through on the basis of
the observations $\sequence{Y}[n][\nset]$ only, which is generally far from
straightforward due to the fact that the observation process
$\sequence{Y}[n][\nset]$ is, on the contrary to $\sequence{Z}[n][\nset]$,
generally non-Markovian. Of particular interest are the \emph{hidden Markov
  models} (HMMs) (alternatively termed \emph{state-space models} in the case
where $\Xset$ is continuous), which constitute a special case of the POMMs in
which the process $\sequence{X}[n][\nset]$ is itself a Markov chain, referred to
as the \emph{state process}, and the observations are conditionally independent 
given the states, such that the marginal conditional distribution of each observation
$Y_n$ depends only on the corresponding state $X_n$. The use of unobservable
states provides the HMMs with a relatively simple dependence structure which is
still generic enough to handle complex, real-world time series in a large
variety of scientific and engineering disciplines (such as financial
econometrics \cite{hull:white:1987,mamon:elliott:2007}, speech recognition
\cite{juang:rabiner:1991}, biology \cite{churchill:1992}, neurophysiology
\cite{fredkin:rice:1987}, etc.; see also the monographs
\cite{macdonald:zucchini:2009} and \cite{cappe:moulines:ryden:2005} for
introductive and state of the art treatments of the topic, respectively), and
the POMMs can be viewed as a natural extension and generalization of this model
class.

In this paper, we will consider a parameterized family of POMMs with parameter
space $\Theta$, where the latter is assumed to be furnished with some metric. For each
$\theta \in \Theta$, the dynamics of the model is specified by the transition
kernel $\Tk[\theta]$ of $\sequence{Z}[n][\nset]$ on $\Xset \times \Yset$, and we
will in this work restrict ourselves to the \emph{fully dominated case} where
$\Tk[\theta]$ has a transition density $\Td[\theta]$ w.r.t. some dominating
measure (all these objects will be defined rigorously in the next
section). Each transition kernel $\Tk[\theta]$ is assumed to have a unique
invariant distribution $\pi_\theta$.

We assume that we have access to a single observation trajectory
$\sequence{Y}[n][\nset]$ sampled from the canonical law $\PP$ induced by
$\Tk[\thv]$ and some initial distribution $\eta_\star$ on $\Xset \times \Yset$,
where $\thv \in \Theta$ is a distinguished parameter interpreted as the
\emph{true} parameter and $\eta_\star$ is generally different from
$\pi_\thv$. In order to estimate $\thv$ via the observations we adopt a
Bayesian framework and introduce a possibly improper \emph{prior}
distribution $\prior$ on $\Theta$, reflecting our a priori belief concerning
$\theta$, and compute the conditional---\emph{posterior}---distribution
$\post{\chunk{Y}{1}{n}}$ of $\theta$ given the observations $\chunk{Y}{1}{n} =
(Y_1, \ldots, Y_n)$, which is, for measurable $A \subseteq \Theta$ and
$\chunk{y}{1}{n} \in \Yset^n$, given by
\begin{equation*} 
    \post{\chunk{y}{1}{n}}(A) = \frac{\int_A p_\theta(\chunk{y}{1}{n}) \, \prior(\rmd \theta)}{\int_{\Theta} p_\theta(\chunk{y}{1}{n}) \, \prior(\rmd \theta)},
\end{equation*}
where $p_\theta(\chunk{y}{1}{n})$ denotes the density of $\chunk{Y}{1}{n}$ given $\theta$. In this general setting, we examine the asymptotics of the posterior distribution and identify model conditions under which the \emph{posterior consistency}
\begin{equation} \label{eq:posterior:consistency}
    \PP \lr{\post{\chunk{Y}{1}{n}}  \weakconv{n \to \infty} \delta_\thv}=1\eqsp,
\end{equation}
holds true, where $\weakconv{}$ denotes weak convergence and $\delta_\thv$
denotes the Dirac mass located at the true parameter $\thv$. In~\eqref{eq:posterior:consistency}, 
$\PP$ denotes the distribution of $\sequence{Y}[n][\nset]$ corresponding 
to the true parameter $\thv$; in this sense we adopt, by 
proving~\eqref{eq:posterior:consistency}, a frequentist point 
of view for the asymptotic behavior of the posterior distribution. 
However, establishing that the influence of the prior is 
overwhelmed by the data as the sample size $n$ grows to infinity 
is of fundamental interest in Bayesian analysis.

\subsection{Previous work}

From the frequentist inference point of view, POMMs have been subjected to extensive research during the last decades. For the important subclass formed by HMMs with finite state space, the asymptotic consistency of the \emph{maximum likelihood estimator} (MLE) was established by \cite{baum:petrie:1966,petrie:1969} and \cite{leroux:1992} in the cases of finite and general observation spaces, respectively, and these results were generalized gradually to more general HMMs in \cite{douc:matias:2002,douc:moulines:ryden:2004,catalot:laredo:2006}. The first MLE consistency result for general HMMs with possibly non-compact state space was obtained in \cite{douc:moulines:olsson:vanhandel:2011}, and \cite{douc:moulines:2012} extended further this result to misspecified models.  For POMMs that fall outside the HMM class, \cite{douc:moulines:ryden:2004} established the MLE consistency for autoregressive models with Markov regimes by applying strong mixing assumptions requiring typically the state space of the latent Markov chain to be compact. Recently, \cite{douc:doukhan:moulines:2013} established the MLE consistency for general observation-driven time series with possibly non-compact state space and \cite{douc-roueff-sim2015b} established the analogous result for general partially dominated POMMs, covering the observation-driven models as a special case. The latter work can be viewed as the state of the art when it concerns MLE analysis for POMMs. The mentioned works demonstrate a variety of techniques, but share the assumption that the parameter space $\Theta$ is a compact set. 

On the other hand, when it concerns Bayesian asymptotic analysis of POMMs, there are only a handful results of which all treat exclusively HMMs. In the case of HMMs with a finite state space, \cite{gassiat:rousseau:2014} provides the posterior consistency (with rates) for parametric models with an unknown number of states and the recent paper \cite{vernet:2015} deals with posterior concentration in the non-parametric case. For more general HMMs, \cite{degunst:shcherbakova:2008} establishes, along the now classical lines of \cite[Theorem~8.3]{lehmann:casella:1998}, a Bernstein-von~Mises-type result under the assumption that the model satisfies, first, a law of large numbers for the log-likelihood function, second, a central limit theorem for the score function and, third, a law of large numbers for the observed Fisher information. As these asymptotic properties, which are the cornerstones of the proof of the asymptotic normality of the MLE, can be established for models satisfying the strong mixing assumption (see \cite{douc:moulines:ryden:2004} and \cite[Chapter~12]{cappe:moulines:ryden:2005}), the result holds, in principle, true for HMMs with a compact state space. A more direct approach to the posterior consistency for HMMs is taken in \cite{papavasiliou:2006}, where the author works with a large deviation bound for the observation process; nevertheless, the analysis is driven by very restrictive model assumptions in terms of strong mixing and additive observation noise.

In conclusion, all available results on posterior concentration for parametric HMMs rest on very stringent model assumptions, in the sense that the state space of the unobservation process is assumed to be compact. Needless to say, this is not the case for many models met in practical applications (such as the linear Gaussian state-space models). In addition, the mentioned results require, without exception, also the parameter space to be compact, which is generally a severe restriction for the Bayesian. Consequently, a general posterior consistency result for parametric POMMs (and, in particular, parametric HMMs) has hitherto been lacking. In the light of the widespread and ever increasing interest in Bayesian inference in models of this sort, which is boosted by novel achievements in computational statistics (especially in the form of \emph{particle} \cite{gordon:salmond:smith:1993} and \emph{particle Markov chain Monte Carlo methods} \cite{andrieu:doucet:holenstein:2010}), this is indeed remarkable, and the goal of the present paper is to fill this gap.

\subsection{Our approach}

In this paper, we establish the posterior consistency \eqref{eq:posterior:consistency} under very mild assumptions which can be checked for a large class of POMMs used in practice. The result is stated in Theorem~\ref{thm:noncompact} for general POMMs with positive transition density (Theorem~\ref{thm:main:result} deals with the case of a compact parameter space) and in Theorem~\ref{thm:hmm} for HMMs under an alternative set of assumptions requiring, e.g., only the emission density to be positive. The starting point of our analysis is the observation that it is, by the Portmanteau lemma, enough to show that for all $A_p = \{ \theta \in \Theta : d(\theta, \thv) \geq 1/p \}$, $p \in \nsetpos$,
\begin{equation} \label{eq:sufficient:cond:consistency}
\limsup_{n \to \infty} \post{\chunk{Y}{1}{n}}(A_p) = 0, \quad \PP\mbox{-a.s.}
\end{equation}
(see Remark~\ref{rem:post:const} for details). Now, by expressing the posterior as
\begin{equation} \label{eq:intro:posterior:alt}
\post{\chunk{Y}{1}{n}}(A) = \frac{\int_A p_\theta(\chunk{Y}{1}{n})  / p_{\thv}(\chunk{Y}{1}{n}) \, \prior(\rmd \theta)}{\int_{\Theta} p_\theta(\chunk{Y}{1}{n})  / p_{\thv}(\chunk{Y}{1}{n}) \, \prior(\rmd \theta)},
\end{equation}
we conclude that \eqref{eq:sufficient:cond:consistency} will hold if
\begin{itemize}
\item[--] all closed sets $A$ not containing $\thv$ are $\PP$-\emph{remote} from $\thv$ in the sense that the numerator
of \eqref{eq:intro:posterior:alt} tends to zero exponentially fast under $\PP$;
\item[--] for all $\delta > 0$, there exists some subset $\Theta_\delta$ of $\Theta$ which is charged by the prior, i.e., $\lambda(\Theta_\delta) > 0$, and such that for all $\theta \in \Theta_\delta$, the ratio $p_\theta(\chunk{Y}{1}{n})  / p_{\thv}(\chunk{Y}{1}{n})$ is, $\PP$-a.s., eventually bounded from below by $\e^{-\delta n}$. This \emph{asymptotic merging property} forces the numerator of \eqref{eq:intro:posterior:alt} to vanish at a faster rate than the denominator for all $\PP$-remote sets $A$, implying \eqref{eq:sufficient:cond:consistency}.
\end{itemize}

This machinery, which is adopted from \cite{barron:chervish:wasserman:1999} and described generally in Section~\ref{sec:general-setting}, does not require the model under consideration to be a POMM; it is hence of independent interest.  As we will see, the situation of a non-compact parameter space calls for a refined notion of $\PP$-remoteness; indeed, by operating under the assumption that the sequence of posterior distributions is \emph{tight}, it is enough to require $\PP$-remoteness to hold on a sufficiently large compact subset of $\Theta$.

The $\PP$-remoteness, the asymptotic merging property and the tightness of the posterior are the fundamental building blocks of our analysis of the posterior concentration. Interestingly, a key finding of us is that the $\PP$-remoteness is closely related to the MLE consistency; more specifically, in Proposition~\ref{prop:equiv:AMLE} we establish that if all sequences of \emph{approximate MLEs} (see Definition~\ref{def:AMLE}) on some compact subset $K$ of $\Theta$ (with $\prior(K) < \infty$) containing $\thv$ are strongly consistent, then $A \cap K$ is $\PP$-remote for all closed sets $A$ not containing $\thv$.
As mentioned in the literature review above, the MLE can, under the assumption
that the parameter space is compact, be proven to be consistent under very mild
model assumptions satisfied for most fully dominated POMMs, and we will hence
obtain the remoteness for free for a large set of models.

When it concerns the asymptotic merging property, we derive the instrumental bound
\begin{equation} \label{eq:intro:image:bound}
\liminf_{n \to \infty} n^{-1} \log \frac{p_\theta(\chunk{Y}{1}{n})}{p_{\thv}(\chunk{Y}{1}{n})} \geq \liminf_{n \to \infty} n^{-1} \log \frac{\bar{p}_\theta(\chunk{Z}{1}{n})}{\bar{p}_{\thv}(\chunk{Z}{1}{n})},
\end{equation}
where $\bar{p}_\theta(\chunk{z}{1}{n})$, $\chunk{z}{1}{n} \in \Zset^n$, denotes the density of the \emph{complete data} $\chunk{Z}{1}{n}$ given $\theta$ (see Lemma~\ref{lem:image-density-lemma} for a more general formulation). In the stationary mode, i.e., when $\eta_\star = \pi_\thv$, the right hand side of \eqref{eq:intro:image:bound} tends, by Birkhoff's ergodic theorem, to minus the expectation $\Delta(\thv, \theta)$ of the Kullback-Leibler divergence (KLD) between $\Tk[\thv]$ and $\Tk[\theta]$ under the stationary distribution. As a consequence, the asymptotic merging property holds true as long as the prior is \emph{information dense} at $\thv$ in the sense that $\prior(\{ \theta \in \Theta : \Delta(\thv, \theta) \geq \delta \}) > 0$ for all $\delta > 0$. This condition can however be checked straightforwardly in general, since $\Delta(\thv, \theta)$ involves a KLD between perfectly known \emph{transition kernels}. Without access to the bound \eqref{eq:intro:image:bound}, an alternative strategy would have been to study directly the limit of the left hand side of \eqref{eq:intro:image:bound} by, e.g., going to ``the infinite past'' in the spirit of \cite{douc:moulines:ryden:2004,douc:moulines:2012}; however, this approach would require the analysis of an expected KLD between \emph{ergodic limits} $p_\theta(Y_0 \mid Y_{- \infty:-1})$ and $p_\thv(Y_0 \mid Y_{- \infty:-1})$ (we refer to the mentioned works for the meaning of these quantities) under the stationary distribution, which is infeasible in general. 

As described above, our technique of handling models with a non-compact parameter space is based on the assumption that the sequence of posterior distributions is tight. Recalling that our objective is the establishment of the gradual concentration of these distributions around the true parameter as $n$ increases, we may expect this assumption to be mild. Indeed, by operating in the stationary mode using Kingman's subadditive theorem, we are able to derive handy assumptions under which the posterior tightness holds at an exponential rate (see Theorem~\ref{thm:noncompact} and Proposition~\ref{prop:hors:compact}). As far as is known to us, this is the first result ever in this direction for models of this sort, and we believe that a similar approach may be used also for extending existing results on MLE consistency to the setting of a non-compact parameter space.

We remark that Birkhoff's ergodic theorem and Kingman's subadditive theorem require the observation process to be stationary (i.e. $\eta_\star = \pi_\thv$). Nevertheless, if for all parameters $\theta$ and initial distributions $\eta$, the distribution of $Y_{1:\infty}$, when initialized according to $\eta$ and evolving according to $\theta$, admits a positive density w.r.t. the distribution of $Y_{1:\infty}$ under the stationary distribution $\pi_\theta$, one may prove that any property that holds a.s. under the latter distribution holds a.s. under the former as well. That such positive densities exist can be established for POMMs in general under the assumption that the transition density is positive (Lemma~\ref{lem:radon}) and for HMMs in particular under the weaker assumption that the emission density is positive and the hidden chain is geometrically ergodic (Lemma~\ref{lem:radon:hmm}), and we are consequently able to treat also the non-stationary case. As far as is known to us, this efficient approach to non-stationarity results has never been taken before.

Finally, we demonstrate the flexibility of our results by checking carefully our assumptions on a partially observed linear Gaussian Markov model as well as the widely used stochastic volatility model proposed by \cite{hull:white:1987} (the latter falls into the framework of HMMs).

To sum up, our contribution is fourfold, since we
\begin{itemize}
\item[--] establish the posterior consistency for very general POMMs under mild assumptions, which allow the state space of the latent part of the model to be non-compact and which can be checked for a large number of models used in practice.
\item[--] link, via the concept of $\PP$-remoteness, the posterior consistency to the consistency of the MLE.
\item[--] are able to treat also the case of a non-compact parameter space.
\item[--] treat efficiently the case of non-stationary observations.
\end{itemize}

The paper is structured as follows. In Section~\ref{sec:part-observ-mark}, we
introduce the POMM framework under consideration and state, in
Section~\ref{sec:main:results}, our main results
(Theorems~\ref{thm:main:result}--\ref{thm:noncompact}) and assumptions. In particular,
we provide an alternative set of assumptions that are taylor-made for the
special case of HMMs. Section~\ref{sec:examples} treats the two examples
mentioned previously and discusses generally our assumptions in the light of
nonlinear state-space models. In Section~\ref{sec:gener-approach} we embed the
problem of posterior concentration into the general framework outlined above,
serving as a machinery for the proofs of our main results. The latter proofs
are found in Section~\ref{sec:proof-main-results} and in
Section~\ref{sec:conclusion} we conclude the paper.


\section{Fully dominated partially observed Markov models (fdPOMM)}

\label{sec:part-observ-mark}

\subsection{Setting}
\label{sec:setting-part-observ-mark}
Let $(\Xset,\Xsigma)$ and $(\Yset,\Ysigma)$ be general measurable spaces
referred to as \emph{state space} and \emph{observation space},
respectively. The product space $\Zset=\Xset\times\Yset$ is then endowed with
the product $\sigma$-field $\Zsigma=\Xsigma\tensprod\Ysigma$, and we set
$\Omega=\Zset^{\nset}$ and $\mcf=\Zsigma^{\tensprod\nset}$. Let further
$\sequence{Z}[n][\nset]$, $\sequence{X}[n][\nset]$ and $\sequence{Y}[n][\nset]$
denote the canonical processes taking on values in the spaces
$(\Zset,\Zsigma)$, $(\Xset,\Xsigma)$ and $(\Yset,\Ysigma)$, respectively, and
defined by $Z_n(\omega)=(x_n,y_n)$, $X_n(\omega)=x_n$ and $Y_n(\omega)=y_n$,
where $\omega=\sequencen{(x_k,y_k)}[k\in\nset]$.  Now, for all $n \in
\nsetpos$, only $\sequence{Y}[n][\nsetpos]$ is observable, and for this reason
we refer to the model as \emph{partially observed}. Let us define
$\mcf_n=\sigma(\chunk{Y}{1}{n})$ for all $n \in \nset$, with $\chunk Y1n =
(Y_1,\ldots, Y_n)$ serving as our general notation for vectors. In addition,
let $(\Theta, \mct)$ be a measurable space and let $\{\Tk[\theta],
\theta\in\Theta\}$ be a collection of Markov transition kernels on $(\Zset,
\Zsigma)$. Denote, for all $\theta \in \Theta$, by $\PP^\theta_\Xinit$ the law
of the canonical Markov chain $\sequence{Z}[n][\nset]$ induced by the Markov
transition kernel $\Tk[\theta]$ and the initial distribution $\Xinit$. We say
that the model is \emph{fully dominated} if there exist two $\sigma$-finite
measures $\mu$ and $\nu$ on $(\Xset,\Xsigma)$ and $(\Yset,\Ysigma)$,
respectively, such that for all $\theta\in\Theta$ and $z\in\Zset$, the
probability measure $\Tk[\theta](z,\cdot)$ is dominated by $\mu\tensprod\nu$,
and we denote by $\Td[\theta](z,\cdot)$ the corresponding density. We may now
introduce the main class of models studied in this article.
\begin{definition}\label{def:fdPOMM}
  We say that $\sequence{Y}[n][\nsetpos]$ follows a \emph{fully dominated partially
  observed Markov model} (fdPOMM) if, for all $n \in \nsetpos$, under the previous
  definitions and assumptions, the distribution of $\chunk Y1n$ is given by
  $\PP^\theta_\eta(\chunk Y1n\in\cdot)$ for some $\theta\in\Theta$ and initial
  distribution $\eta$ on $\Zset$.
\end{definition}
We denote, with a slight abuse of notation, by $\chunk y1n\mapsto\py{\theta}{\eta}{\chunk y1n}$
the density of $\chunk Y1n$ with respect to $\nu^{\tensprod n}$ under
$\PP_\eta^\theta$, i.e.,
$$
 \py{\theta}{\eta}{\chunk y1n} \eqdef \int
\Td[\theta](z_0,(x_{1},y_1)) \prod_{\ell = 1}^{n-1}
\Td[\theta]((x_\ell,y_\ell),(x_{\ell+1},y_{\ell+1})) \,
\Xinit(\rmd z_0) \, \mu^{\tensprod n}(\rmd \chunk x1n) \eqsp.
$$
Now, let $\prior$ be some measure, called the \emph{prior distribution}, on $(\Theta,\mct)$. We will always assume
that $\Theta$ is endowed with some metric $d$ and that $\mct$ is taken to be the corresponding
Borel $\sigma$-field.

Given observations $\chunk{y}{1}{n} \in \Yset^n$, the \emph{posterior distribution}
associated with the initial probability distribution
$\Xinit$ is defined by
\begin{equation}
  \label{eq:post-def-pomc}
\post{\chunk{y}{1}{n}}(A)\eqdef\frac{\int_A \py{\theta}{\Xinit}{\chunk{y}{1}{n}} \, \prior(\rmd
  \theta)}{\int_{\Theta}\py{\theta}{\Xinit}{\chunk{y}{1}{n}} \, \prior(\rmd
  \theta)} \eqsp, \quad \mbox{for all }A\in\mct\eqsp.
\end{equation}
For the numerator and denominator of \eqref{eq:post-def-pomc} to be
well-defined, we will always assume that
$(\theta,z,z')\mapsto\Td[\theta](z,z')$ is measurable on
$\Theta\times\Zset^2$. However, at this point it is not guaranteed that the
ratio itself is well-defined (and does not degenerate into $0/0$ or
$\infty/\infty$). In fact, we will only be interested in the case where
$\post{\chunk{Y}{1}{n}}$ is $\PP\as$ a probability distribution for $n$ large
enough, where $\PP$ denotes the \emph{true distribution} of $\sequence{Y}[n][\nset]$ and
is introduced below.

We always assume the following.
\begin{hyp}{B}
\item \label{ass:stat:dist}
For all $\theta\in\Theta$, the Markov transition kernel $\Tk[\theta]$ has a unique
stationary distribution $\pi_\theta$.
\end{hyp}
Under \refhyp[B]{ass:stat:dist} it is typically assumed that the law of the
observations is given by $\PP=\PP^\thv_{\pi_\thv}$ for some distinguished
parameter $\thv \in \Theta$ interpreted as the \emph{true} parameter (which is
not known a priori). We proceed similarly and set $\piv=\pi_{\thv}$. However,
in the present paper we will also consider the more general case where
$\PP=\PP^\thv_{\Xinitv}$ for some possibly unknown initial distribution
$\Xinitv\neq\piv$, and since the initial distribution $\Xinit$ appearing
in~\eqref{eq:post-def-pomc} is designed arbitrarily by the user, we cannot
assume that $\Xinit=\Xinitv$. See also \Cref{rem:merging} for further comments concerning this. 
\begin{remark}
Under \refhyp[B]{ass:stat:dist}, since $\Tk[\theta]$ is
dominated by $\mu \tensprod \nu$, the stationary probability measure $\pi_\theta$ is
also dominated by $\mu \tensprod \nu$, and by abuse of notation, we still denote
by $\pi_\theta$ the associated density.
\end{remark}

\subsection{Main results}
\label{sec:main:results}
We now state the main results of this contribution, which consist in providing
general sufficient conditions for the posterior consistency
$$
\PP\lr{\post{\chunk{Y}{1}{n}} \weakconv{n \to \infty} \delta_\thv}=1\eqsp,
$$
where $\basicwc$ denotes weak convergence and  $\delta_\theta$ denotes a Dirac point
mass located at $\theta$.
The proof of this result is based on basically two main ingredients. The first is to
ensure that only parameters $\theta$ close to $\thv$ have a large likelihood
as the number $n$ of observations tends to infinity.
This is formalized by the following assumption.
\begin{hyp}{B}
  \item\label{ass:amle:pomc} If $K \in \mct$ is a compact set
    containing $\thv$, then all $K$-valued, $\sequence{\mcf}[n][\nset]$-adapted random sequences $\sequence{\hat \theta}[n][\nsetpos]$ such that for all $n \in \nsetpos$, 
$$
n^{-1} \log \py{\hat{\theta}_n}{\eta}{\chunk Y1n}
\geq n^{-1} \log \py{\thv}{\piv}{\chunk Y1n}  +\epsilon_n \eqsp\mbox{,} 
$$
with
$$
\lim_{n\to\infty}\epsilon_n=0\mbox{,} \quad\PP^\thv_{\piv}\mbox{-a.s.,}
$$
converges $\PP^\thv_{\piv}\as$ to $\thv$. 
\end{hyp}
For identifiable models, this assumption follows directly from standard
consistency results on the maximum likelihood estimator for ergodic partially
observed Markov chains; see \cite{douc-roueff-sim2015b} and the references
therein. In other words, it does not require a specific treatment from a
Bayesian point of view.

\begin{remark}
  Note that in the case where $\Theta$ is compact, \refhyp[B]{ass:amle:pomc} only needs to be checked for $K=\Theta$.
\end{remark}

The second ingredient is to ensure that the prior distribution does not concentrate
around parameters whose likelihood is too small asymptotically.
We provide below sufficient conditions that are easily checked in two specific
situations.

\paragraph{The fully dominated case.}

For all $\theta \in \Theta$ we set
\begin{equation}
  \label{eq:defDelta}
  \D(\thv, \theta) \eqdef 
  \PE^\thv_\piv\left[\kullback{\Tk[\thv](Z_0,\cdot)}{\Tk[\theta](Z_0,\cdot)}\right]\in[0,\infty]\eqsp,
\end{equation}
where for any two probability measures $P$ and $Q$ defined on the same probability
  space, $\kullback{P}{Q}$ denotes the KLD 
  of $Q$ from
  $P$ defined by
$$
\kullback{P}{Q} \eqdef
\begin{cases}
\int \log \frac{\rmd P}{\rmd Q} \, \rmd P\eqsp, & \mbox{if $P \ll Q$}\eqsp, \\
\infty\eqsp, & \mbox{otherwise}\eqsp.
\end{cases}
$$

\begin{theorem} \label{thm:main:result}
Consider a \fdPOMM\ satisfying~\refhyp[B]{ass:stat:dist}[ass:amle:pomc] with a compact parameter space $\Theta$.
Let $\prior$ be a finite measure on $\Theta$ and $\Xinit$ an arbitrary
distribution on $(\Zset,\Zsigma)$.
Then the conditions
\begin{hyp}{B}
\item \label{ass:q:positive} for all $\theta \in \Theta$,  $\Td[\theta](z,z')>0$ $\mu \tensprod \nu$\mbox{-a.s.},
\end{hyp}
\begin{hyp}{B}
\item \label{ass:kullback}
for all $\delta>0$,
$\displaystyle
\prior\left(\ens{\theta\in \Theta}{\D(\thv, \theta)\leq\delta}\right) >0$,
\end{hyp}
imply that for all initial distributions $\Xinitv$ on $(\Zset, \Zsigma)$,
$$
\PP^\thv_\Xinitv\lr{\post{\chunk{Y}{1}{n}}  \weakconv{n \to \infty} \delta_\thv}=1\eqsp.
$$
\end{theorem}
The proof of this result can be found in \Cref{sec:proof-crefthm:main}.

\paragraph{An alternative set of conditions for HMMs.}
The HMMs can be viewed as a
subclass of the {\fdPOMM}s defined by the following assumption.
\begin{hyp}{C}
\item \label{ass:fully:dominated:restricted}
For all $\theta \in \Theta$, $z = (x,y) \in \Zset$ and $z = (x',y') \in \Zset$,
$$
\Td[\theta](z, z')=\Tdx[\theta](x,x')\Tdy[\theta](x',y')\eqsp\text{,}
$$
where $\Tdx[\theta]$ and $\Tdy[\theta]$ are kernel densities on
$\Xset\times\Xset$ and $\Xset\times\Yset$, respectively.
\end{hyp}
Under \ \refhyp[C]{ass:fully:dominated:restricted} \ we denote by
$\Tx[\theta]$ the Markov transition kernel on $(\Xset,\Xsigma)$ associated with
the transition density $\Tdx[\theta]$. In this subclass of models it may happen
that the positiveness condition \refhyp[B]{ass:q:positive} does not hold, but
only since $\Tdx[\theta]$ vanishes. In this case, we rely on the following
weaker assumption.
\begin{hyp}{C}
\item \label{ass:g:positive} For all $\theta \in \Theta$ and $x\in\Xset$,  $\Tdy[\theta](x,\cdot)>0$ $\nu$\mbox{-a.s.}
\end{hyp}
In this context, we define
\begin{equation}
  \label{eq:def-dhmm}
\Dhmm(\thv,\theta) \eqdef \int
\kullback{\Tdy[\thv](x,\cdot)}{\Tdy[\theta](x',\cdot)} \,
\piv(\rmd x\times\Yset) \, \pi_\theta(\rmd x'\times\Yset)
\end{equation}
and consider the following assumption replacing \refhyp[B]{ass:kullback}.
\begin{hyp}{C}
\item \label{ass:kullback:hmm}
For all $\delta>0$,
$\prior(\{ \theta \in \Theta : \Dhmm(\thv, \theta) \leq \delta \}) >0$.
\end{hyp}
Finally, we impose the following condition.
\begin{hyp}{C}
\item \label{ass:TVforgetting} For all $\theta \in \Theta$ and all initial
  distributions $\Xinit$ on $(\Xset,\Xsigma)$, $(\Xinit\Tx[\theta]^n)_{n \in \nsetpos}$ converges to the
  first marginal of $\pi_\theta$ in the total variation norm, i.e.,
$$
\lim_{n \to \infty}\|\Xinit\Tx[\theta]^n -\pi_\theta(\cdot\times\Yset)\|_{\mathrm{TV}}=0\eqsp.
$$
\end{hyp}
The kernel notation in~\refhyp[C]{ass:TVforgetting} is standard and described in detail in Section~\ref{sec:kernel:notation}.

We can now state the following result, whose proof can be found in \Cref{sec:proof-crefthm:hmm}.
\begin{theorem}\label{thm:hmm}
  Theorem~\ref{thm:main:result} holds still true when~\refhyp[B]{ass:q:positive}[ass:kullback]
  are replaced
  by~\refhyp[C]{ass:fully:dominated:restricted}[ass:TVforgetting].
\end{theorem}
\begin{remark}
  \label{rem:iid-kullback}
It is interesting to note that in the case of i.i.d. observations, which
corresponds to the HMM case \refhyp[C]{ass:fully:dominated:restricted} with
$\Tdx[\theta](x,x')$ arbitrary (say equal to 1 with $\mu$ an arbitrary
probability measure) and $\Tdy[\theta](x',y')$ not depending on $x'$,
simply denoted by  $\Tdy[\theta](y')$ hereafter, we have
$$
\D(\thv, \theta)=\Dhmm(\thv,\theta)=\kullback{\Tdy[\thv]}{\Tdy[\theta]}\;.
$$
Hence \refhyp[B]{ass:kullback} and \refhyp[C]{ass:kullback:hmm} boil down to
the well known condition of the i.i.d. setting introduced by
\cite{schwartz:1965}, see also \cite[Eqn.~(1)]{ghosal:1997}.
\end{remark}
\paragraph{Non-compact parameter space.}

Needless to say, a drawback of \Cref{thm:main:result,thm:hmm} is that $\Theta$ is assumed to be compact.
This assumption is standard in the frequentist setting, e.g., when studying the
maximum likelihood estimator, but can be problematic in the Bayesian setting,
where it is often convenient to work with priors defined on non-compact spaces for
computational reasons. We now derive some additional conditions dealing with
the non-compact case.

Define for all $A\in \mct$, $n \in \nsetpos$ and $\chunk{y}{0}{n} \in \Yset^{n+1}$,
\begin{equation}
  \label{eq:psup-def}
\psup[A]{\chunk{y}{0}{n}} \eqdef \sup_{(\theta,x_0) \in A \times \Xset} \int \prod_{\ell = 0}^{n-1} \Td[\theta](z_\ell,z_{\ell+1}) \, \mu^{\tensprod n}(\rmd \chunk{x}{1}{n})
\end{equation}
(with $z_\ell=(x_\ell,y_\ell)$).
Now the following result holds true also for a possibly non-compact $\Theta$.
\begin{theorem}\label{thm:noncompact}
  Consider a \fdPOMM\ satisfying~\refhyp[B]{ass:stat:dist} and~\refhyp[B]{ass:amle:pomc}. Let $\prior$ be a Radon measure on $\Theta$ and $\Xinit$ an arbitrary distribution on
  $(\Zset,\Zsigma)$. In addition, suppose that the following conditions hold true.
\begin{hyp}{B}
\item\label{ass:tight:bound} There exist $\ell \in\nsetpos$ and a
  non-decreasing sequence $\sequence{C}[m][\nset]$ of compact sets in $\mct$ such that
\begin{align}
& \limsup_{m \to \infty}\psup[C_m^c]{\chunk{Y}{0}{\ell}}=0 \quad \PP^\thv_\piv\mbox{-a.s.,} \label{eq:conv:zero}\\
& \PE^\thv_\piv[\log^+\psup[\Theta]{\chunk{Y}{0}{\ell}}] <\infty\eqsp. \label{eq:log:moment}
\end{align}
\end{hyp}
\begin{hyp}{B}
\item \label{ass:additive}
There exists $n_0 \in \nsetpos$ such that
\begin{align}\label{eq:B6-1}
& \int \prior(\rmd \theta) \, \py{\theta}{\Xinit}{\chunk{Y}{1}{n_0}} <\infty \quad \PP^\thv_\piv\mbox{-a.s.,}  \\
\label{eq:B6-2}
&\PE^\thv_\piv[\log \py{\thv}{\piv}{Y_{n_0} \mid \chunk{Y}{1}{n_0-1}}]>-\infty\eqsp .
\end{align}
\end{hyp}
Then \refhyp[B]{ass:q:positive}[ass:kullback] or
\refhyp[C]{ass:fully:dominated:restricted}[ass:TVforgetting] imply that for all initial distributions
  $\Xinitv$ on $(\Zset,\Zsigma)$,
$$
\PP^\thv_\Xinitv\lr{\post{\chunk{Y}{1}{n}}  \weakconv{n \to \infty} \delta_\thv}=1\eqsp.
$$
\end{theorem}
The proof is postponed to \Cref{sec:proof-crefthm:tightness}. In~\eqref{eq:B6-2}, $\py{\thv}{\piv}{Y_{n_0} \mid \chunk{Y}{1}{n_0-1}}$ is, as usual, defined as the ratio
$\py{\thv}{\piv}{\chunk{Y}{1}{n_0}}/\py{\thv}{\piv}{\chunk{Y}{1}{n_0-1}}$,
with the convention that the denominator is unity if $n_0=1$. Note that this
ratio is always well defined under \refhyp[B]{ass:q:positive} or
\refhyp[C]{ass:fully:dominated:restricted}[ass:g:positive].
\begin{remark}\label{rem:B6-n0-n}
  It is interesting to observe that, as detailed in
  \Cref{lemma:noncompactcond-n0-all-n}, each condition in
  \refhyp[B]{ass:additive} implies the same condition with $n_0$ replaced by
  any $n \in \nset$ larger than $n_0$. It is therefore sufficient to check the conditions independently with
  two possibly different $n_0$. The fact that~\eqref{eq:B6-1} holds for all
  $n\geq n_0$ is of particular interest, since it guaranties that both the
  numerator and denominator in the definition of the posterior
  $\post{\chunk{Y}{1}{n}}$ in~\eqref{eq:post-def-pomc} are finite. On the other
  hand, by \refhyp[B]{ass:q:positive} or
  \refhyp[C]{ass:fully:dominated:restricted}[ass:g:positive] the
  denominator is positive. Hence, if~\eqref{eq:B6-1} holds, 
  the posterior $\post{\chunk{Y}{1}{n}}$ is well defined as a
  probability distribution for $n$ large
  enough.
\end{remark}

\section{Examples}
\label{sec:examples}

\subsection{Partially observed Gaussian linear Markov model}

First, we consider a linear Gaussian \fdPOMM\ defined on $\Zset=\rset^p\times\rset^q$ by
\begin{align}
Z_{k+1}=\Phi_\theta Z_k+ \epsilon_{k+1}\eqsp, \label{eq:poglm}
\end{align}
where $\Phi_\theta$ is $(p+q)\times(p+q)$ matrix and $\sequence{\epsilon}[n][\nsetpos]$ is a
sequence of i.i.d. centered Gaussian vectors with $(p+q)\times(p+q)$ covariance
matrix $R_\theta$.  In the following we assume that $\Theta$ is a compact subset of $\rset^d$
and that for all $\theta\in\Theta$, $\Phi_\theta$ has spectral radius strictly
less than unity and $R_\theta$ is positive definite. Then
$\sequence{Z}[n][\nset]$ is a vector auto-regressive process with transition density  
$$
q_\theta(z_0,z_1)=(2\pi)^{-(p+q)/2} (\mathrm{det}(R_\theta))^{-1}\exp \left( -\frac12(z_1-\Phi_\theta z_0)^\intercal R_\theta^{-1}(z_1-\Phi_\theta z_0) \right),
$$
with $(z_0, z_1) \in \Zset^2$, satisfying {\refhyp[B]{ass:stat:dist}}. This framework includes 
the widespread \emph{linear Gaussian state-space model}
\begin{equation} \label{eq:glspm}
\begin{split}
X_{k+1}&=A_\theta X_k+\zeta_{k+1}, \\
Y_k&=B_\theta X_k + \xi_k,
\end{split}
\quad k \in \nset,
\end{equation}
corresponding to
$$
\Phi_\theta=\begin{bmatrix}A_\theta&0\\B_\theta A_\theta&0\end{bmatrix}\quad\text{and}\quad
\epsilon_k=\begin{bmatrix}\zeta_{k}\\ B_\theta\zeta_{k}+\xi_k\end{bmatrix}.
$$
Note that the model \eqref{eq:glspm} is an HMM only if $\zeta_k$ and $\xi_k$
are uncorrelated for all $k$, which is not assumed in the model~\eqref{eq:poglm}.
 Assumption \refhyp[B]{ass:q:positive} is trivially satisfied.
The expected KLD in~\eqref{eq:defDelta} is easily computed; indeed, for all
$\thv=(\Phi_*,R_*)$ and $\theta=(\Phi,R)$ in $\Theta$,
\begin{multline*}
\D(\thv, \theta) = \frac12 \left[ \operatorname{tr}(R^{-1}R_*)-p-q-\log\mathrm{det}(R^{-1}R_*) \right. \\
\left. +\mathrm{tr}(R^{-1}(\Phi-\Phi_*)^\intercal (\Phi-\Phi_*) \Gamma_*) \right] ,
\end{multline*}
where
$$
\Gamma_* \eqdef \PE^\thv_\piv\left[Z_0 Z_0^\intercal \right] = \sum_{k = 0}^\infty \Phi_*^k R(\Phi_*^\intercal)^k.
$$
It follows that $\theta\mapsto\D(\thv, \theta)$ is continuous at $\thv$ (where
it always vanishes) and thus that {\refhyp[B]{ass:kullback}} is satisfied whenever
$\prior$ is strictly positive on $\Theta$ (i.e., $\prior(A) > 0$ for every
non-empty open set $A$). 
\Cref{thm:main:result} hence applies as soon as Assumption~\refhyp[B]{ass:amle:pomc} holds true (examples are treated in, e.g., \cite{catalot:laredo:2006} or \cite[Section~3.3]{douc:moulines:olsson:vanhandel:2011}). 

\subsection{General HMMs}
In this section, we consider the case of HMMs, i.e., we assume
\refhyp[C]{ass:fully:dominated:restricted}. Up to our knowledge, the following set of assumptions, which are borrowed from \cite[Theorem~1]{douc:moulines:olsson:vanhandel:2011}, are the weakest available for obtaining the strong consistency of the
(approximate) MLE for well-specified HMMs.
\begin{hyp}{D}
\item \label{ass:posit:Harris}
For all $\theta \in \Theta$, the Markov kernel $\Tx[\theta]$ is aperiodic positive Harris recurrent.
\end{hyp}
Note that under \refhyp[C]{ass:fully:dominated:restricted}, this assumption is
sufficient for \refhyp[B]{ass:stat:dist}, i.e., the existence of a unique stationary probability 
measure for each complete chain kernel $\Tk[\theta]$, $\theta\in\Theta$.
\begin{hyp}{D}
\item \label{ass:kernel-likelihood}
$\PE_\piv^\thv[\sup_{x\in\Xset}\,(\log \Tdy[\thv](x,Y_0))^+]<\infty$,
$\PE_\piv^\thv\lrb{\left|\log\int \Tdy[\thv](x,Y_0)\,\pi_\thv(\rmd x)\right|}<\infty$.
\end{hyp}

\begin{hyp}{D}
\item \label{ass:localBound}
For all $\theta\neq\thv$, there are a neighborhood $\mcu_\theta$ of
$\theta$ such that
$$
	\sup_{\thetapr\in\mcu_\theta}\sup_{(x,x') \in \Xset^2}\Tdx[\thetapr](x,x')<\infty\eqsp\mbox{,}
	\qquad
	\PE_\piv^\thv\lrb{\sup_{\thetapr \in \mcu_\theta}
	\sup_{x\in\Xset} (\log \Tdy[\thetapr](x,Y_0))^+}
	<\infty 
$$
and an integer $r_\theta$ such that
$$
	\PE_\piv^\thv\lrb{\sup_{\thetapr \in \mcu_\theta}
	(\log \py{\thetapr}{\Xinit}{\chunk{Y}{1}{r_\theta}})^+}
	<\infty\eqsp.
$$
\end{hyp}

\begin{hyp}{D}
\item \label{ass:localUnifInteg}
For all $\theta \neq \thv$ and $n \geq r_\theta$, the function
$\thetapr \mapsto \py{\thetapr}{\Xinit}{\chunk{Y}{1}{n}}$ is
upper-semicontinuous at $\theta$, $\PP_\piv^\thv$-a.s.
\end{hyp}

\begin{hyp}{D}
\item \label{ass:mixing}
For all $\theta \neq \thv$ such that
$\py{\theta}{\Xinit}{\chunk{Y}{1}{r_\theta}}>0$ $\PP_\piv^\thv$-a.s., we have
$$
	\liminf_{n\to\infty}\PP_\piv^\thv(\chunk Y1n\in A_n)>0\eqsp\mbox{,}\qquad
	\limsup_{n\to\infty}n^{-1}\log\PP_{\Xinit}^\theta (\chunk Y1n\in A_n)<0	
$$
for some sequence of sets $A_n\in\Ysigma^{\tensprod(n+1)}$.
\end{hyp}
As a consequence of \Cref{thm:main:result,thm:hmm}, the only
requirement on the prior $\prior$ in the case of general HMMs with 
a compact parameter space is given by \refhyp[B]{ass:kullback} or \refhyp[C]{ass:kullback:hmm}, 
depending on the positivity assumption on the kernel densities
(\refhyp[B]{ass:q:positive} or  \refhyp[C]{ass:g:positive}).

\begin{theorem}
\label{thm:consist:general:hmm}
Assume \refhyp[C]{ass:fully:dominated:restricted} and
\refhyp[D]{ass:posit:Harris}[ass:mixing] with a compact parameter space
$\Theta$. If \refhyp[B]{ass:q:positive}[ass:kullback] or
\refhyp[C]{ass:g:positive}[ass:kullback:hmm] hold, then, for all initial
distributions $\Xinitv$ on $(\Zset, \Zsigma)$,
$$
\PP^\thv_\Xinitv\lr{\post{\chunk{Y}{1}{n}}  \weakconv{n \to \infty} \delta_\thv}=1\eqsp.
$$
\end{theorem}

\begin{proof}
  Assumption \refhyp[D]{ass:posit:Harris} implies \refhyp[B]{ass:stat:dist} and
  \refhyp[C]{ass:TVforgetting} (see
  \cite[Theorem~13.0.1]{meyn:tweedie:1993}). Moreover, the strong consistency
  of the AMLE has been proved in \cite[Theorem
  1]{douc:moulines:olsson:vanhandel:2011} under
  \refhyp[D]{ass:posit:Harris}[ass:mixing], showing that
  \refhyp[B]{ass:amle:pomc} is satisfied. Thus, all the assumptions of
  \Cref{thm:main:result} and \Cref{thm:hmm} are satisfied and the proof follows.
\end{proof}

Thus, for HMMs, one may, in order to apply \Cref{thm:consist:general:hmm}, choose to 
check \ {\refhyp[B]{ass:q:positive}[ass:kullback]} or
 {\refhyp[C]{ass:g:positive}[ass:kullback:hmm]}, depending on the model.
Consider for example the nonlinear state-space model on 
$\Xset=\rset^p$,
\begin{align} \label{eq:ARCH}
	X_{k + 1} = T_\theta(X_k) + \Sigma_\theta (X_k) \zeta_{k + 1} \eqsp, \quad k \in \nset, 
\end{align}
where $\theta$ is an $m$-dimensional parameter on a compact space $\Theta
\subset \rset^m$, $(\zeta_k)_{k \in \nsetpos}$ is an i.i.d.\ sequence of
$d$-dimensional random vectors with density $\rho_\zeta$ \wrt\ Lebesgue measure
$\lleb$ and $T_\theta : \rset^p \to \rset^p$ and $\Sigma_\theta : \rset^p \to
\rset^{p \times p}$ are given (measurable) matrix-valued functions.  Conditions
on $T_\theta$, $\sigma_{\theta}$, $\rho_{\theta}$ and $g_\theta$ to ensure
\refhyp[D]{ass:posit:Harris}[ass:mixing] can be found in
\cite[Section~3.3]{douc:moulines:olsson:vanhandel:2011}.  These conditions are
stated in the case where $\rho_\zeta$ is positive over $\rset^d$ but can be
easily adapted to the case where the support of $\rho_\zeta$ is
compact. However, in the latter case, \refhyp[B]{ass:q:positive} does not hold,
and assuming \refhyp[C]{ass:g:positive}, we can rely on
\refhyp[C]{ass:kullback:hmm} as an alternative for \refhyp[B]{ass:kullback}. In
what follows, we explain how to deal with $\pi_\theta$ appearing in the
definition \eqref{eq:def-dhmm} of $\Dhmm(\thv,\theta)$ using standard
conditions.

Assume that there exist a measurable function $V: \Xset \to [1,\infty)$ and
constants $(C,\rho) \in \rsetpos \times (0,1)$ such that for all $n \in \nset$
and all $(x,x') \in \Xset^2$,
\begin{equation}
\label{eq:V-geometric-ergodicity}
\sup_{\theta \in \Theta} \vnorm{\Tx[\theta]^n(x,\cdot)-\Tx[\theta]^n(x',\cdot)}{V} \leq C(V(x)+V(x')) \rho^n\eqsp,
\end{equation}
where for any signed measure $\chi$ on $(\Xset,\Xsigma)$, $\vnorm{\chi}{V}
\eqdef \sup \chi f$, where the supremum is taken over all measurable functions
$f: \Xset \to \rset$ such that $|f|_V\eqdef \mu-\mathrm{esssup}_{x \in \Xset}[ |f(x)|/V(x) ]
\leq 1$ and $\chi f$ denotes the integral of $f$ w.r.t. $\chi$ (see
Section~\ref{sec:kernel:notation} for details).

\begin{proposition}
Assume that \eqref{eq:V-geometric-ergodicity} holds. Moreover, suppose that
\begin{enumerate}[label=(\roman*),resume=Vgeo]
\item \label{item:Vgeo:kullbackassumption} for all $\thv\in\Theta$, there
  exists $C_\star>0$ such that, for $\mu\pp\; x \in \Xset$ and $\mu\pp\;
  x' \in \Xset$,
  $\theta\mapsto\kullback{\Tdy[\thv](x,\cdot)}{\Tdy[\theta](x',\cdot)}$ is
  continuous at $\thv$ and
$$
\sup_{\theta\in\Theta}\kullback{\Tdy[\thv](x,\cdot)}{\Tdy[\theta](x',\cdot)}\leq
C_\star \;V(x)V(x');
$$
\item \label{item:Vgeo:one} there exists a constant $M<\infty$ such that for
  $\mu\pp\; x \in \Xset$, $\sup_{\theta \in \Theta}\Tx[\theta]V(x)\leq M V(x)$;
\item \label{item:Vgeo:two} for $\mu\pp\; x \in \Xset$, $\lim_{\theta \to \thv} \vnorm{\Tx[\theta](x,\cdot)-\Tx[\thv](x,\cdot)}{V} =0$.
\end{enumerate}
Then, \refhyp[C]{ass:kullback:hmm} is satisfied for all strictly positive prior measures $\lambda$ on $(\Theta, \mathcal{T})$.
\end{proposition}

\begin{proof}
  Let $\thv\in\Theta$. It is sufficient to show that the function $\theta
  \mapsto \Dhmm(\thv, \theta)$ is continuous at $\thv$, where it takes on the value zero. 
  For all $\theta \in \Theta$, denote by $\pix_\theta$ the marginal
  probability measure on $(\Xset,\Xsigma)$ defined by
  $\pix_\theta(A)=\pi_\theta(A \times \Yset)$ for all $A\in\Xsigma$, and let the function
  $\phi_\theta: \Xset \to \rsetpos$ be defined by
$$
\phi_\theta(x') \eqdef \int \kullback{\Tdy[\thv](x,\cdot)}{\Tdy[\theta](x',\cdot)} \pix_{\thv}(\rmd x)\eqsp.
$$
We may then write 
\begin{equation} \label{eq:dhmm}
\Dhmm(\thv, \theta)=[\pix_\theta \phi_\theta - \pix_\thv \phi_\theta]+\pix_\thv \phi_\theta 
\end{equation}
(where, as usual, $\pix_\theta \phi_\theta$ denotes the integral of $\phi_\theta$ w.r.t. $\pix_\theta$; see Section~\ref{sec:kernel:notation}). 
We proceed stepwise. 

\noindent\textbf{Step~1.} We first show that $\pix_\thv V <\infty$. Denoting
 by ${\cal M}_V$ the Banach space of signed measures $\mu$ such that
$|\mu| V <\infty$, equipped with the $V$-norm, we will
actually prove the following more precise assertion.
\begin{enumerate}[label=(\roman*),resume=Vgeo]
\item \label{item:Vgeo:assertion:uniform:Vnormconv} For all $x\in\Xset$,
$\Tx[\theta]^n(x,\cdot)$ converges to $\pi_\theta^X$ in ${\cal M}_V$, uniformly
over $\theta\in\Theta$.
\end{enumerate}
(Hence $\pix_\theta V <\infty$ for all $\theta \in \Theta$.)  For all
probability measures $\mu_1$, $\mu_2$ on $(\Xset,\Xsigma)$, all $\theta \in
\Theta$ and all $f$ such that $|f|_V \leq 1$, we have
\begin{multline*}
|\mu_1\Tx[\theta]^nf-\mu_2\Tx[\theta]^nf|=\left|\int \mu_1(\rmd x) \, \mu_2(\rmd x') \, \lr{\Tx[\theta]^nf(x)-\Tx[\theta]^nf(x')} \right|
\\ \leq \int \mu_1(\rmd x) \, \mu_2(\rmd x') \, \vnorm{\Tx[\theta]^n(x,\cdot)-\Tx[\theta]^n(x',\cdot)}{V} \eqsp.
\end{multline*}
Thus, \eqref{eq:V-geometric-ergodicity} provides a constant $C>0$ such that 
\begin{equation*}
 \vnorm{\mu_1\Tx[\theta]^n-\mu_2\Tx[\theta]^n}{V} \leq C(\mu_1 V+\mu_2V) \rho^n \eqsp.
\end{equation*}
Taking $\mu_1=\delta_x$ and $\mu_2=\Tx[\theta](x,\cdot)$ and combining with
\ref{item:Vgeo:one}, we get that 
$$
\sum_{n = 0}^\infty \sup_{\theta \in \Theta}\vnorm{\Tx[\theta]^n(x,\cdot)-\Tx[\theta]^{n+1}(x,\cdot)}{V}\leq
\frac{C (1+M)}{1-\rho} V(x), 
$$
and since ${\cal M}_V$ is complete, we obtain that $\Tx[\theta]^n(x,\cdot)$
converges uniformly in ${\cal M}_V$ over $\Theta$. Denoting by $\tilde
\pi_\theta$ this limit, we are only required to show that
$\tilde\pi_\theta=\pi_\theta^X$ in order to
establish~\ref{item:Vgeo:assertion:uniform:Vnormconv}.  Now, since all bounded
functions have finite $V$-norm, for all $A\in\Xsigma$,
$$
\tilde \pi_\theta(A)=\lim_{n \to \infty} \Tx[\theta]^n\1{A}(x)=\lim_{n \to \infty} \Tx[\theta]^{n+1}\1{A}(x)=\tilde \pi_\theta \Tx[\theta](A),
$$
so that $\tilde \pi_\theta$ is an invariant probability measure for
$\Tx[\theta]$. It is thus $\pi_\theta^X$ as a consequence of \refhyp[B]{ass:stat:dist}.
  
\noindent\textbf{Step~2.} Next, we show that
$\lim_{\theta\to\thv}\pix_\thv \phi_\theta =0$. Assumption~\ref{item:Vgeo:kullbackassumption},
$\pix_\thv V <\infty$ and the dominated convergence theorem give immediately that $\theta
\mapsto \phi_\theta(x)$ is continuous at $\thv$, for $\mu\pp\; x$. Moreover, $\sup_{\theta\in \Theta}
|\phi_\theta|_V<\infty$. Consequently, using again the dominated
convergence theorem, the last term on the right hand side of \eqref{eq:dhmm} converges to
$\pix_\thv \phi_\thv =0$ as $\theta$ tends to $\thv$.

\noindent\textbf{Step~3.} Finally, we consider the term between brackets
in~\eqref{eq:dhmm} and show that it converges to zero as $\theta\to\thv$. Since we
just showed that $\sup_{\theta\in \Theta} |\phi_\theta|_V<\infty$, it suffices
to show that $\pix_\theta$ converges to $\pix_{\thv}$ in $\mathcal{M}_V$.
By~\ref{item:Vgeo:assertion:uniform:Vnormconv}, this boils down to proving that
for all $n\in \nsetpos$, $\Tx[\theta]^n(x,\cdot)$ converges to $\Tx[\thv]^n(x,\cdot)$
in $\mathcal{M}_V$. This can be done by induction on $n$. The base case $n=1$
corresponds to \ref{item:Vgeo:two}. The induction follows easily
from the following
decomposition, valid for all $f:\Xset\to\rset$ such that $|f|_{V}\leq1$: 
\[
\begin{split}
\lefteqn{\left|\Tx[\theta]^{n+1}(x,f)-\Tx[\thv]^{n+1}(x,f)\right|} \hspace{10mm} \\
&\leq \left|\Tx[\theta]^{n}(x,\Tx[\theta]f) - \Tx[\thv]^{n}(x,\Tx[\theta]f )\right|+\left|\Tx[\thv]^{n}(x,\Tx[\theta]f - \Tx[\thv]f )\right|
\\
&\leq\vnorm{\Tx[\theta]^{n}-\Tx[\thv]^{n}}{V}\,\sup_{\theta\in\Theta}\vnorm{\Tx[\theta]}V
+\vnorm{\Tx[\thv]^{n}}{V}\,
\vnorm{\Tx[\theta]-\Tx[\thv]}{V}.  
\end{split}
\]
By observing that~\ref{item:Vgeo:one} implies
$\sup_{\theta \in \Theta} \vnorm{\Tx[\theta]^n}V<\infty$ for all $n \in \nsetpos$, we may conclude the proof. 
\end{proof}

\subsection{Stochastic volatility models}

Consider the stochastic volatility model 
\begin{equation} \label{eq:stoVol}
\begin{split}
X_{k+1}&=\varphi X_k+\sigma \zeta_{k + 1} \eqsp, \\
Y_k&=\beta \exp(X_k/2) \epsilon_k \eqsp, 
\end{split} 
\quad k \in \nset, 
\end{equation}
where $(\zeta_k)_{k \in \nsetpos}$ and $(\epsilon_k)_{k \in \nset}$ are
independent sequences of i.i.d. Gaussian random vectors in $\rset^2$ with zero
mean and identity covariance matrix; see \cite{hull:white:1987}. A general
description of stochastic volatility models as HMMs is provided in
\cite{genon-catalot:jeantheau:laredo:2000}. In this case, $\Xset = \Yset =
\rset$. If the parameter vector $\theta=(\beta,\sigma,\varphi)$ belongs to a
compact parameter space, we may apply the theory developed in the previous
section. However, in this example, $\theta$ is assumed to belong to the
\emph{non-compact} parameter space
$$
\Theta\eqdef\{(\beta,\sigma,\varphi):  \beta\geq \beta_-,\; \sigma \geq \sigma_-,\;|\varphi|\leq \varphi_+\}\eqsp,
$$
where $\beta_->0$, $\sigma_->0$ and $\varphi_+ \in (0,1)$. Denote by $\thv=(\beta_\star,\sigma_\star,\varphi_\star)$ the true value of the parameter. In this model,
\begin{equation} \label{eq:stoVol:g:theta}
\begin{split}
\Tdx[\theta](x,x') &=\frac{1}{\sqrt{2 \pi \sigma^2}} \exp\left\{-\frac{(x'-\varphi x)^2}{2 \sigma^2}\right\} \eqsp, \\
\Tdy[\theta](x,y) &=\frac{1}{\sqrt{2 \pi \beta^2}} \exp\left\{-\frac x2-\frac{y^2}{2 \beta^2}\rme^{-x}\right\} \eqsp.
\end{split}
\end{equation}
Assumption~\refhyp[B]{ass:stat:dist} is clearly satisfied with
\begin{equation} \label{eq:stoVol:pi}
\pi_\theta(B)=\int \frac{1}{2\pi}\sqrt{\frac{1-\varphi^2}{\sigma^2}}
\exp\lr{-\frac{(1-\varphi^2)x^2}{2\sigma^2}-\frac{u^2}{2}} \oneSub{B}(x,\beta
\rme^{x/2}u)  \, \rmd u \, \rmd x \eqsp.
\end{equation}
Assumption~\refhyp[B]{ass:amle:pomc} follows from  \cite[Section
3.3]{douc:moulines:olsson:vanhandel:2011} and
Assumption~\refhyp[B]{ass:q:positive} is immediate.
Moreover, straightforward algebra yields
\begin{multline*}
\D(\thv, \theta) = \log\frac{\sigma\beta}{\sigma_\star \beta_\star}+
\frac12\PE^\thv_\piv\left[X_1^2\right] \left(\sigma^{-2}-\sigma_\star^{-2}\right)+
\PE^\thv_\piv\left[X_0X_1\right]\left(\frac{\phi_\star}{\sigma_\star^2}-\frac{\phi}{\sigma^2}\right)\\+
\frac12\PE^\thv_\piv\left[X_0^2\right]\left(\frac{\phi^2}{\sigma^2}-\frac{\phi_\star^2}{\sigma_\star^2}\right)+
\frac12\PE^\thv_\piv\left[Y_1^2\rme^{-X_1}\right]\left(\beta^{-2}-\beta_\star^{-2}\right).
\end{multline*}
Note that $\PE^\thv_\piv[\rme^{a|X_0|}]$ and $\PE^\thv_\piv[|Y_0|^a]$ are
finite for all $a>0$. Hence by the Cauchy-Schwarz inequality all the
expectations appearing in the previous display are finite and we conclude that,
for all $\thv\in\Theta$, $\theta\mapsto\D(\thv, \theta)$ is a continuous
function. Thus, \refhyp[B]{ass:kullback} holds for all
priors being strictly positive (possibly unnormalized) measures  $\prior$ on $(\Theta, \mathcal{T})$.
Since $\Theta$ is non-compact, the posterior consistency needs to be established via 
\Cref{thm:noncompact}. To this end, it only remains to check
\refhyp[B]{ass:tight:bound} and~\refhyp[B]{ass:additive}.

We check \refhyp[B]{ass:tight:bound} with $\ell=2$. Write, for all $A \in \mct$ and $\chunk{y}{0}{2} \in \rset^3$,
\begin{equation}
\psup[A]{\chunk{y}{0}{2}}=\sup_{(\theta,x_0) \in A \times \rset} D_{\theta,x_0}(\chunk{y}{0}{2})\eqsp, \label{eq:stoVol:def:psup}
\end{equation}
with
$$
D_{\theta,x_0}(\chunk{y}{0}{2}) \eqdef \iint  \Tdx[\theta](x_0,x_1) \Tdy[\theta](x_1,y_1) \Tdx[\theta](x_1,x_2) \Tdy[\theta](x_2,y_2) \, \rmd x_1 \, \rmd x_2\eqsp. \nonumber
$$
We will use the following bounds obtained by straightforward algebra: for all $\theta \in \Theta$,
  \begin{equation} \label{eq:stoVol:easyBounds-1}
   \sup_{x \in \rset} \Tdy[\theta](x,y) =\frac{1}{|y|\sqrt{2 \pi \rme}}\eqsp,
    \quad \int \Tdy[\theta](x,y) \, \rmd x=\frac{1}{|y|}\eqsp,
   \quad \Tdy[\theta](x,y)\leq \frac{\rme^{-x/2}}{\beta\sqrt{2\pi}}
   \end{equation} 
   and 
   \begin{equation} \label{eq:stoVol:easyBounds-2}
   \sup_{(x,x') \in \rset^2} \Tdx[\theta](x,x')=\frac{1}{\sqrt{2 \pi \sigma^2}}\eqsp. 
  \end{equation}
Then, using \eqref{eq:stoVol:easyBounds-1} and \eqref{eq:stoVol:easyBounds-2},
\begin{multline}
\label{eq:stoVol:one}
D_{\theta,x_0}(\chunk{y}{0}{2}) \leq
\sup_{(x_0, x_1) \in \rset^2}\Tdx[\theta](x_0,x_1) \int \Tdy[\theta](x_1,y_1)
 \int \Tdx[\theta](x_1,x_2) \, \rmd x_2  \, \rmd x_1 \sup_{x_2 \in \rset}\Tdy[\theta](x_2,y_2)\\
\leq \frac{1}{|y_1| |y_2| \sqrt{2\pi\sigma^2} \sqrt{2 \pi  \rme }}. 
\end{multline}
Moreover, using the definitions of $\Tdy[\theta](x_1,y_1)$, $\Tdx[\theta](x_1,x_2)$ and the bound on $\Tdy[\theta](x_2,y_2)$ given in  \eqref{eq:stoVol:easyBounds-1}, we get, by standard calculations,
   \begin{align}
D_{\theta,x_0}(\chunk{y}{0}{2}) &\leq \left(\int \Tdx[\theta](x_0,x_1) \, \rmd x_1\right) \sup_{x_1 \in \rset} \lr{\Tdy[\theta](x_1,y_1) \int \Tdx[\theta](x_1,x_2) \Tdy[\theta](x_2,y_2) \, \rmd x_2} \nonumber\\
& \leq \sup_{x_1 \in \rset}
\left(\frac{1}{\sqrt{2 \pi \beta^2}} \exp\left\{-\frac{x_1}{2}-\frac{y_1^2}{2
      \beta^2}\rme^{-x_1}\right\} \right. \nonumber\\
&\phantom{\leq  \sup_{x_1 \in \rset} (}\left. \times \frac{1}{\sqrt{2 \pi \sigma^2}}
\frac{1}{ \sqrt{2 \pi \beta^2}}\int \exp\left\{-\frac{(x_2-\varphi x_1)^2}{2 \sigma^2}-\frac{x_2}{2}\right\} \rmd x_2\right) \nonumber\\
&= \sup_{x_1 \in \rset} \left( \frac{1}{\sqrt{2 \pi\beta^2}} \exp\left\{-\frac{x_1}{2}-\frac{y_1^2}{2 \beta^2}\rme^{-x_1}\right\} \frac{1}{\sqrt{2 \pi\beta^2}} \exp\left\{-\frac{1}{2}\varphi x_1 +\frac{\sigma^2}{8}\right\} \right) \nonumber\\
&= \sup_{x_1 \in \rset} \left( \frac{1}{2 \pi \beta^2} 
\exp\left\{-\frac{x_1}{2}(1+\varphi)-\frac{y_1^2}{2 \beta^2}\rme^{-x_1}
  +\frac{\sigma^2}{8}\right\} \right) \nonumber\\
&=\frac{\rme^{\frac{\sigma^2}{8}}}{2\pi \beta^{1-\varphi}} \lrb{\frac{(1+\varphi)\rme^{-1}}{y_1^2}}^{\frac{1+\varphi}{2}}\eqsp. \label{eq:stoVol:two}
   \end{align}
   We then set $C_m=\{\theta \in \Theta:  \sigma^2 \leq \log m, \, \beta\leq \rme^m\}$, so that
$$
C_m^c \subset \{\theta \in \Theta: \, \sigma^2 > \log m\} \cup \{\theta \in \Theta: \, \sigma^2 \leq \log m\eqsp, \,\beta> \rme^m\}\eqsp.
$$
Combining this inclusion with \eqref{eq:stoVol:def:psup}, \eqref{eq:stoVol:one} and \eqref{eq:stoVol:two} yields
$$
\limsup_{m \to \infty}\psup[C_m^c]{\chunk{Y}{0}{2}}=0 \quad \PP^\thv_\piv\as,
$$
implying that \eqref{eq:conv:zero} is satisfied with $\ell=2$. Now, by \eqref{eq:stoVol:one}, for all $\chunk{y}{0}{2} \in \rset^3$,
\begin{multline*}
\log^+ \psup[\Theta]{\chunk{y}{0}{2}}=\sup_{(\theta, x_0) \in \Theta \times \rset} \log^+ D_{\theta,x_0}(\chunk{y}{0}{2}) \\ 
\leq \frac 12 \log^+(4\pi^2\sigma^2 \rme) + \log^+ |y_1|+ \log^+ |y_2|, 
\end{multline*}
and using \eqref{eq:stoVol:pi}, this implies that \eqref{eq:log:moment} is
satisfied with $\ell=2$. Thus, we may conclude that \refhyp[B]{ass:tight:bound} holds.

Condition~\eqref{eq:B6-1} in \refhyp[B]{ass:additive} holds if $\prior$ is a
probability measure on $(\Theta, \mathcal{T})$. Alternatively, one may, e.g., 
use~\eqref{eq:stoVol:one} and~\eqref{eq:stoVol:two} to obtain that for all $y_{1:2} \in \rset^2$, 
$$
 \py{\theta}{\Xinit}{\chunk{y}{1}{2}}\leq  \left(
 \frac{1}{|y_1| |y_2| \sqrt{2\pi\sigma^2} \sqrt{2 \pi  \rme }} \right)
\wedge
\left(\frac{\rme^{\frac{\sigma^2}{8}}}{2\pi \beta^{1-\varphi}} \lrb{\frac{(1+\varphi)\rme^{-1}}{y_1^2}}^{\frac{1+\varphi}{2}}\right).
$$
Condition~\eqref{eq:B6-1} in \refhyp[B]{ass:additive} is then implied by less restrictive condition 
$$
\int \sigma^{-1} \wedge(\rme^{\sigma^2/8}/\beta^{1-\varphi_+}) \, \prior(\rmd\theta) < \infty.
$$

We now prove that \eqref{eq:B6-2} in Assumption~\refhyp[B]{ass:additive} holds true with
$n_0 = 1$. By Jensen's inequality and \eqref{eq:stoVol:g:theta},
\begin{align*}
\PE^\thv_\piv[\log \py{\thv}{\piv}{Y_1}]&=\PE^\thv_\piv\lrb{\log \iint \piv(\rmd x_0) \, \Tdx[\thv](x_0,x_1)\Tdy[\thv](x_1,Y_1)}\\
&\geq \PE^\thv_\piv\lrb{\iint \piv(\rmd x_0) \, \Tdx[\thv](x_0,x_1)\log \Tdy[\thv](x_1,Y_1)}\\
&= -\frac12 \log(2\pi\beta_\star^2)-\frac{\PE^\thv_\piv[X_1]}{2} -\frac{\PE^\thv_\piv[Y_1^2\rme^{-X_1}]}{2\beta_\star^2}>-\infty\eqsp,
\end{align*}
where we used again that $\PE^\thv_\piv[\rme^{a|X_0|}]$ and $\PE^\thv_\piv[|Y_0|^a]$
are both finite for all $a > 0$. Finally, \Cref{thm:noncompact} applies, establishing the
posterior consistency for the model \eqref{eq:stoVol}, e.g., for all
strictly positive probability measures $\prior$ on $\Theta$ and all initial
distributions $\Xinitv$ on $(\Zset, \Zsigma)$,
$$
\PP^\thv_\Xinitv\lr{\post{\chunk{Y}{1}{n}}  \weakconv{n \to \infty} \delta_\thv}=1 \eqsp.
$$


\section{A general approach to posterior consistency}
\label{sec:gener-approach}

It turns out to be convenient to embed the problem of posterior consistency for {\fdPOMM}s into a more general setting. This widened perspective allows a number of universal steps to be identified, by which the posterior consistency can be established in general. Importantly, this machinery is not at all specific to
the framework of {\fdPOMM}s and is thus of independent interest.

\subsection{General setting}
\label{sec:general-setting}
Let $(\Omega,
\mcf,\sequence{\mcf}[n][\nset] ,\PP)$ be a filtered probability space. If for all
$n \in \nset$, $\nu_n$ is a $\sigma$-finite measure on $(\Omega,\mcf_n)$ and the restriction $\PP|_{\mcf_n}$ of $\PP$ to $\mcf_n$ is
absolutely continuous \wrt\ $\nu_n$, then we say that $(\Omega,\mcf,\sequence{\mcf}[n][\nset],\sequence{\nu}[n][\nset] ,\PP)$ is a {\em progressively dominated filtered probability space}.

Moreover, let $\{\gd{\theta}{n},\,\theta\in \Theta\}$ be a collection of probability
densities \wrt\ $\nu_n$. Let $\gdstar{n}$ denote the Radon-Nikodym derivative
$$
\gdstar{n} \eqdef \frac{\rmd \PP|_{\mcf_n}}{\rmd \nu_n}\eqsp.
$$
In \Cref{sec:proof-main-results}, it is shown how the {\fdPOMM}s can be cast
into this general setting; see also \Cref{sec:exple-iid} for a treatment of the
more simple i.i.d. case.

We now introduce the prior and posterior distributions, denoted
by $\lambda$ and $\lambda_n$, respectively, in this general setting.
Let $\lambda$ be a non-zero $\sigma$-finite measure on $(\Theta, \mct)$. Then,
for all $A \in \mct$, the posterior ``probability'' $\lambda_n(A)$ with prior
$\lambda$ is defined by
\begin{equation}
  \label{eq:def-posterior-proba}
\lambda_n(A) \eqdef \frac{\int_A \lambda(\rmd \theta) \, \gd{\theta}{n}}{\int_\Theta \lambda(\rmd \theta) \, \gd{\theta}{n}}
\end{equation}
whenever this ratio is well-defined. In what follows, we will always assume that $(\theta, \omega)
\mapsto \gd{\theta}{n}(\omega)$ is measurable from
$(\Theta\times\Omega,\mct\tensprod\mcf_n)$ to $(\rset_+,\borel{\rset_+})$, which implies that the numerator and denominator are (non-negative)
$\mcf_n$-measurable random variables.

We now introduce the main assumption on the model. It says that there are
``sufficiently many'' $\theta$ for which the likelihood ratio
$\gd{\theta}{n}/\gdstar{n}$ is not decreasing exponentially fast under $\PP$.

\begin{hyp}{A}
\item\label{ass:general} For all $\delta > 0$, there exists
    $\Theta_\delta \in \mct$ such that $\lambda(\Theta_\delta) > 0$
    and for all $\theta \in \Theta_\delta$,
    \begin{equation*}
        \liminf_{n \to \infty} n^{-1} \log \frac{\gd{\theta}{n}}{\gdstar{n}}
        \geq - \delta  \quad \PP\as
    \end{equation*}
\end{hyp}
We expect the set $\Theta_\delta$ in \refhyp[A]{ass:general} to contain
parameters $\theta$ whose corresponding densities $\gd{\theta}{n}$ are ``asymptotically close'' to the
true ones $\gdstar{n}$.
In contrast, the following definition addresses the case of parameters
indexing densities remaining ``far away'' from the true ones.

\begin{definition}\label{def:admissible}
We say that a set $A \in \mct$ is $\PP$-\emph{remote} if and
only if
$$
         \limsup_{n \to \infty} n^{-1} \log
        \left( \int_A \frac{\gd{\theta}{n}}{\gdstar{n}}\;
        \prior(\rmd \theta)   \right) < 0 \quad  \PP\as
$$
Moreover, we say that a set $A$ is \emph{approximately}
$\PP$-\emph{remote} if and only if for all $\varepsilon > 0$ there exists a
set $K_\varepsilon \in \mct$ such that
\begin{enumerate}[(i)]
    \item \label{item:ass:Compact} $A\cap K_\varepsilon$ is
      $\PP$-remote;
    \item \label{item:ass:outCompact}
    $
        \displaystyle \limsup_{n \to \infty} \prior_n(K_\varepsilon^c)
        \leq \varepsilon  \quad  \PP\as
    $
\end{enumerate}
We will denote  by $\mathcal{A}_{\PP}$ the class of all approximately $\PP$-remote sets.
\end{definition}
Here ``remote'' refers to the fact that the likelihood ratio averaged over the
parameters within $A$ decreases exponentially fast to zero under $\PP$.
\begin{remark}\label{rem:tightness}
  Typically, for all $\varepsilon>0$, Property~\ref{item:ass:outCompact} in
  \Cref{def:admissible} is satisfied for a well chosen compact set
  $K_\varepsilon$.  We refer to this as the \emph{uniform} $\PP\as$
  \emph{tightness property of the posterior distribution}.  In this case a set
  $A$ is approximately $\PP$-remote whenever $A\cap K$ is $\PP$-remote for all
  compact sets $K$.
\end{remark}
Let
$\PP_{\theta,n}$ be the probability on $(\Omega, \mcf_n)$ defined by
$$
\PP_{\theta,n}(B) \eqdef \int_B \gd{\theta}{n} \, \rmd \nu_n, \quad B \in \mcf_n\eqsp.
$$
We have the following characterization of $\PP$-remote sets, which is closely
related to \cite[Theorem~5(2)]{barron:1988}, although in a simplified form and without
relying on the asymptotic-merging condition, thanks to the normalization by $\gdstar{n}$.
\begin{proposition}\label{prop:cns:remote} The set $A \in \mct$ is $\PP$-remote
  if and only if there exists a sequence $(B_n)_{n \in \nset}$ of sets in
  $\mcf$ such that $B_n\in\mcf_n$ for all $n\in\nset$ and
\begin{align}
& \limsup_{n \to \infty}  n^{-1} \log \int_A \prior(\rmd \theta) \, \PP_{\theta,n}(B_n)<0\eqsp, \label{eq:cns:theta}\\
& \PP\lr{\liminf_{n \to \infty} B_n} =1\eqsp.  \label{eq:cns:thv}
\end{align}
\end{proposition}
\begin{proof}
See \Cref{subsec:proof:prop:cns:remote}.
\end{proof}

Following the approach of \cite{barron:chervish:wasserman:1999} we have the following  general result, which
extends the i.i.d. framework used in that reference.
\begin{theorem} \label{thm:general:post:const}
    Assume that \refhyp[A]{ass:general} holds.
    Then all approximately $\PP$-remote sets $A \in \mathcal{A}_{\PP}$ satisfy
\begin{equation} \label{eq:conv:lambda:remote}
    \lim_{n \to \infty} \prior_n(A) = 0  \quad \PP\as\eqsp,
\end{equation}
where $\lambda_n(A)$ is defined by~(\ref{eq:def-posterior-proba}).
\end{theorem}

\begin{proof}
See \Cref{subsec:proof:general:post:const}.
\end{proof}

\begin{remark}
It may happen that the ratio that defines $\lambda_n(A)$
  in~\eqref{eq:def-posterior-proba} is not well-defined. However, in
  \eqref{eq:conv:lambda:remote} it should be understood that $\lambda_n(A)$ is
  well-defined, $\PP\as$, for $n$ large enough (otherwise the limit would not
  be defined).
\end{remark}
\begin{remark}\label{rem:true-posterior-consistency}
Note that $\lambda_n$ is, $\PP\as$, a
  well-defined (posterior) probability on $(\Theta,\mct)$ if
  \begin{equation}
    \label{eq:posterior-well-defined-condition}
\PP\lr{0<\int \lambda(\rmd \theta) \gd{\theta}n <\infty}=1\eqsp.
  \end{equation}
  There are simple although restrictive sufficient conditions to
  ensure~\eqref{eq:posterior-well-defined-condition}.  For instance, note that
  when the prior is proper, i.e., when $\lambda$ is a finite measure, then
  $\int_\Theta \lambda(\rmd \theta) \gd{\theta}n <\infty$, $\nu_n$-almost everywhere and thus
  $\PP\as$ Moreover, if for all $\theta \in \Theta$, $\gd{\theta}{n}>0$ $\PP\as$,
  then $\PP(\int_\Theta \lambda(\rmd \theta) \gd{\theta}n>0)=1$.
\end{remark}

From now on we suppose that $(\Theta,d)$ is a metric space, and
  let $\mct$ be the Borel $\sigma$-field.

\begin{remark} \label{rem:post:const}
  Assume that $\lambda_n$ is, $\PP\as$, a well-defined probability measure on $(\Theta,\mct)$ for
   $n$ large enough. Suppose in addition that there exists $\thv\in\Theta$ such that $A_p=\ensemble{\theta\in \Theta}{d(\theta,\thv)\geq 1/p}$ is
  approximately $\PP$-remote for all $p\in \nsetpos$. Then \Cref{thm:general:post:const} implies
  \begin{equation}
    \label{eq:weak-conv-bayesan-consistency}
\PP\lr{\lambda_n \weakconv{n \to \infty} \delta_\thv}=1\eqsp.
  \end{equation}
This stems from the fact that by the Portmanteau lemma, $\{\lambda_n
\weakconv{n} \delta_\thv\}$ is implied by $\{\lim_n \prior_n(A_p) =
0\text{ for all }p \in \nsetpos\}$. Indeed, suppose that the latter event has occurred
and let $F$ be a closed
set. If $\thv \notin F$ then there exists $p\in \nsetpos$ such that $F\subset A_p$
and $\limsup_{n}\lambda_n(F) \leq \limsup_{n} \lambda_n(A_p) =
0=\delta_\thv(F)$. On the other hand, since we assumed $\lambda_n$ to be a probability
measure for $n$ large enough, if $\thv \in F$, $\limsup_n \lambda_n(F)\leq
1=\delta_\thv(F)$.
\end{remark}

The following lemma will be useful for checking \refhyp[A]{ass:general} with an explicit expression of $\Theta_\delta$.

\begin{lemma}\label{lem:image-density-lemma}
  Let
  $(\Omega,\mcf,\sequence{\overline{\mcf}}[n][\nset],\sequence{\overline{\nu}}[n][\nset]
  ,\PP)$ be a progressively dominated filtered probability space. For all $n
  \in \nset$, let $\bgdn{n}$ be a probability density function \wrt\
  $\overline{\nu}_n$.  Denote by $\bgdstar{n}$ the density of
  $\overline{\PP}^*_n=\PP|_{\overline{\mcf}_n}$ \wrt\ $\overline{\nu}_n$ and
  by $\overline{\PP}_n$ the probability measure having density $\bgdn{n}$ \wrt\
  $\overline{\nu}_n$. Suppose that for all $n\in\nset$, $\tilde\mcf_n$ is a
  sub-$\sigma$-field of $\overline{\mcf}_n$ and define by
\begin{equation}
  \label{eq:ratio-image-lemma-new-density}
\tgdn{n} \eqdef \frac{\rmd [{\overline{\PP}_n}_{|\tilde\mcf_n}]}{\rmd
  [{\overline{\nu}_n}_{|\tilde\mcf_n}]}
\quad\text{and}\quad \quad \tgdstar{n} \eqdef \frac{{\rmd[ \overline{\PP}^*_n}_{|\tilde\mcf_n}]}{\rmd[{\overline{\nu}_n}_{|\tilde\mcf_n}]} 
\end{equation}
the Radon-Nikodym derivatives of the $\tilde \mcf_n$-restrictions of $\overline{\PP}_n$ and
  $\overline{\PP}^*_n$ w.r.t. the $\tilde \mcf_n$-restriction of
  $\overline{\nu}_n$, respectively. Then it holds that
\begin{equation}
  \label{eq:ratio-image-lemma}
n^{-1} \log \frac{\tgdn{n}}{\tgdstar{n}} \geq n^{-1} \log
\frac{\bgdn{n}}{\bgdstar{n}}+\epsilon_n \quad\text{with}\quad
  \lim_{n\to\infty}\epsilon_n=0\quad\PP\as
\end{equation}
\end{lemma}

\begin{proof}
Let $n \in \nset$. The result follows from the identity
\begin{equation}
\label{eq:lr-identity}
\frac{\tgdn{n}}{\tgdstar{n}}=
\overline{\PE}^*_n\cond{\frac{\bgdn{n}}{\bgdstar{n}}}{\tilde\mcf_n}\quad \PP^*_n\-\as;
\end{equation}
indeed, since, by \eqref{eq:lr-identity},
$$
\overline{\PP}^*_n\lr{\frac{\tgdn{n}}{\tgdstar{n}}=0,\,\frac{\bgdn{n}}{\bgdstar{n}}>0} = 0 \eqsp,
$$
it holds that
$$
\overline{\PP}^*_n\lr{\frac{\bgdn{n}}{\bgdstar{n}}> n^2 \frac{\tgdn{n}}{\tgdstar{n}}}=
\overline{\PP}^*_n\lr{\frac{\bgdn{n}}{\bgdstar{n}}> n^2
  \frac{\tgdn{n}}{\tgdstar{n}},\,\frac{\tgdn{n}}{\tgdstar{n}}\neq0}  \eqsp.
$$
Hence, using the Markov inequality, we get
$$
\overline{\PP}^*_n\lr{\frac{\bgdn{n}}{\bgdstar{n}}> n^2 \frac{\tgdn{n}}{\tgdstar{n}}}
\leq
n^{-2}
\,\overline{\PE}^*_n\lrb{\lr{\left.{\frac{\bgdn{n}}{\bgdstar{n}}}\right/{\frac{\tgdn{n}}{\tgdstar{n}}}}
\,\indiceTxt{\frac{\tgdn{n}}{\tgdstar{n}}\neq0}}\eqsp.
$$
By conditioning on $\tilde \mcf_n$ and reapplying~\eqref{eq:lr-identity}, we conclude that the expectation on
the right hand side of the previous display is equal to
$\PP^*_n\lr{\frac{\tgdn{n}}{\tgdstar{n}} \neq0}\leq1$.
We thus get that
$$
\overline{\PP}^*_n\lr{\frac{\bgdn{n}}{\bgdstar{n}}> n^2 \frac{\tgdn{n}}{\tgdstar{n}}}
\leq  n^{-2}\eqsp.
$$
Since $\overline{\PP}^*_n$ coincides with $\PP$ on $\overline{\mcf}_n$, the
Borel-Cantelli lemma gives that
$$
\PP\lr{\frac{\bgdn{n}}{\bgdstar{n}}> n^2\, \frac{\tgdn{n}}{\tgdstar{n}}  \quad
  \text{i.o.}} = 0.
$$
This implies~\eqref{eq:ratio-image-lemma} and concludes the proof.
\end{proof}

\subsection{$\PP$-remoteness of $\thv$-missing compact sets }

Note that \Cref{thm:general:post:const} does not rely on the existence of
a \emph{true parameter} $\thv$. This only appears in
\Cref{rem:post:const} where we assume the existence of a
parameter $\thv$ such that closed sets not containing this parameter are
approximately $\PP$-remote. In this section we relate the notion of
$\PP$-remoteness to the consistency of approximate maximum likelihood
estimators. The first step is to relate the true density $\gdstar{n}$ to a true
parameter by assuming that $\gdstar{n}$ and $\gd{\thv}{n}$ \emph{merge with probability one}
in the sense of \cite[Definition 1]{barron:1988}, i.e.,
\begin{equation}
  \label{eq:merge-at-thetastar}
\lim_{n \to \infty} \frac1n\log \frac{\gd{\thv}{n}}{\gdstar{n}}= 0 \eqspeq \PP\as
\end{equation}

Let $A \in \mct$ be a closed set that does not contain
$\thv$. In this section we show that $A\cap K$ is $\PP$-remote for every
compact $K \in \mct$ on which one is able to establish the consistency of approximate maximum
likelihood estimators. For most models of interest, this is actually possible
for all compact sets $K$, and consequently $A$ is approximately $\PP$-remote
whenever the uniform $\PP\as$ tightness property of the posterior distribution
holds true (see \Cref{rem:tightness}).

\begin{definition} \label{def:AMLE}
Let  $K \in \mct$ be compact.
We say that $\sequence{\hat \theta}[n][\nset]$ is a sequence of \emph{approximate maximum likelihood
estimators} (AMLEs) \emph{on} $K$ if it is $\sequence{\mcf}[n][\nset]$-adapted and for all $n\in\nset$, $\hat \theta_n \in K$ and
\begin{align*}
n^{-1}\log \gd{\hat \theta_n}{n} \geq n^{-1}\log \gdstar{n} +\epsilon_n
 \quad\text{with}\quad
  \lim_{n\to\infty}\epsilon_n=0\quad\PP\as
\end{align*}
\end{definition}

\begin{proposition} \label{prop:equiv:AMLE}
Let $K \in \mct$ be compact and suppose
that there exists $\thv\in K$ such that~\eqref{eq:merge-at-thetastar} holds.
Then the two following assertions are equivalent.
\begin{enumerate}[label=(\roman*)]
\item\label{item:1} All sequences $\sequence{\hat \theta}[n][\nset]$ of AMLEs on $K$ are
  strongly consistent.
\item\label{item:2} For all closed sets $A$ not containing $\thv$,
  \begin{align}\label{eq:AMLE-eq}
\limsup_{n \to \infty} \sup_{\theta \in A\cap K} \frac1n \log  \frac{\gd{\theta}{n}}{\gdstar{n}}
< 0  \eqspeq  \PP\as
  \end{align}
\end{enumerate}
Suppose that one of these assertions holds true and, in addition, that $\prior(K)<\infty$.
Then, for all closed sets $A$
not containing $\thv$, the set $A \cap K$ is $\PP$-remote.
\end{proposition}

\begin{proof}
We first show that~\ref{item:1} implies~\ref{item:2}.
Suppose that~\ref{item:2} is false; then there exists a closed set $A$  not
containing $\thv$ such that~\eqref{eq:AMLE-eq} does not hold.
For all $n$, let  $\tilde \theta_n\in A\cap K$ be such that
$$
\log  \frac{\gd{\tilde\theta_n}{n}}{\gdstar{n}}\geq
\sup_{\theta \in A\cap K} \log  \frac{\gd{\theta}{n}}{\gdstar{n}} - 1 \eqsp.
$$
Since~\eqref{eq:AMLE-eq} does not hold, there exists, on a set $\Omega^*$ with
$\PP(\Omega^*)>0$, an increasing sequence
$\sequence{n}[k][\nset]$ such that
$$
\lim_{k \to \infty} \frac1{n_k}\log  \frac{\gd{\tilde\theta_{n_k}}{n_k}}{\gdstar{n_k}}\geq0\eqsp.
$$
Define the random variables $\hat{\theta}_n$ by setting, on $\Omega^*$ and if $n=n_k$ for
some $k\in\nset$, $\hat{\theta}_n=\tilde\theta_n$ and $\hat{\theta}_n=\thv$
otherwise. Then by~\eqref{eq:merge-at-thetastar} and the previous display,
$(\hat\theta_n)_{n \in \nset}$ is a sequence of AMLEs; however, it is not strongly consistent, since
$\hat{\theta}_{n_k}=\tilde\theta_{n_k}\in A\cap K\not\ni\thv$ on $\Omega^*$ for all
$k\in\nset$. Hence~\ref{item:1} does not hold.

We now show that~\ref{item:2} implies~\ref{item:1}. Let $A$ a closed set not containing $\thv$,
and let $\sequence{\hat \theta}[n][\nset]$ be a sequence of AMLEs on $K$.
Then $\lim_{n \to \infty} \epsilon_n = 0$ $\PP\as$ with
$$
\epsilon_n \leq \frac{1}{n} \log  \frac{\gd{\hat \theta_{n}}{n}}{\gdstar{n}}\eqsp.
$$
We thus have that
$$
\hat \theta_n \in A \cap K\Longrightarrow
\epsilon_n \leq \sup_{\theta \in A\cap K} \frac{1}{n} \log  \frac{\gd{\theta}{n}}{\gdstar{n}} \eqsp.
$$
Now~\ref{item:2} and the limit $\lim_{n \to \infty} \epsilon_n = 0$ $\PP\as$ imply that
$$
\epsilon_n > \sup_{\theta \in A\cap K} \frac{1}{n} \log
\frac{\gd{\theta}{n}}{\gdstar{n}}
$$
eventually, $\PP$-a.s. 
The previous implication therefore shows that $\hat \theta_n \in K\setminus A$
eventually, $\PP\as$ The proof is completed by taking $A=\{\theta\in\Theta :
d(\theta,\thv)<1/p\}^c$ for any positive integer $p$, which shows that
$\hat\theta_n$ is strongly consistent.

The last assertion of \Cref{prop:equiv:AMLE} follows immediately from the bound
$$
\log
        \left( \int_{A\cap K} \frac{\gd{\theta}{n}}{\gdstar{n}}\;
        \prior(\rmd \theta)   \right)
\leq \log\prior(K)+\sup_{\theta \in A\cap K} \log  \frac{\gd{\theta}{n}}{\gdstar{n}}\;.
$$
\end{proof}

\subsection{Proof of \Cref{thm:general:post:const} }
\label{subsec:proof:general:post:const}
\newcommand{\ef}[2]{F_{#1,#2}}
We preface the proof of \Cref{thm:general:post:const} by the following lemma.

\begin{lemma} \label{eq:lemma:merging:w:pr:one}
    Under \refhyp[A]{ass:general}, for all $\varepsilon > 0$,
    \begin{equation} \label{eq:merge:one}
    \PP \left( \int \prior(\rmd \theta) \, \frac{\gd{\theta}{n}}{\gdstar{n}}
        \leq \e^{- \varepsilon n} \quad\text{i.o.} \right) = 0 \eqsp.
    \end{equation}
\end{lemma}

\begin{remark}
According to the terminology used in \cite[Definition 1]{barron:1988}, the property \eqref{eq:merge:one} says that $\int \prior(\rmd \theta) \, \gd{\theta}{n}$ and $\gdstar{n}$ merge with probability one.
\end{remark}


\begin{proof}
    Pick $\varepsilon > 0$ and write
    $$
        \e^{\varepsilon n} \int \prior(\rmd \theta) \, \frac{\gd{\theta}{n}}{ \gdstar{n}}
        \geq \int_{\Theta_{\varepsilon/2}} \prior(\rmd \theta) \,
        \ef{\theta}{n} \eqsp,
    $$
    where $\Theta_{\varepsilon/2}$ is defined in \refhyp[A]{ass:general} and for all $\theta \in \Theta$ and $n \in \nsetpos$, 
    $$
        \ef{\theta}{n} \eqdef \exp \left( \varepsilon n
        + \log \frac{\gd{\theta}{n}}{\gdstar{n}} \right) \eqsp.
    $$
    By assumption, for all $\theta \in \Theta_{\varepsilon/2}$,
    $\liminf_n \ef{\theta}{n} = \infty$ $\PP$-a.s., and the proof is
    completed by establishing that, $\PP$-a.s.,
    \begin{equation} \label{eq:diverging:integral}
        \liminf_{n \to \infty} \int_{\Theta_{\varepsilon/2}} \prior(\rmd \theta) \,
        \ef{\theta}{n} = \infty \eqsp.
    \end{equation}
    To check \eqref{eq:diverging:integral}, note that, by Fubini's theorem,
    $$
        \PE \left[ \int_{\Theta_{\varepsilon/2}} \prior(\rmd \theta) \,
        \indiceTxt{\liminf_{n \to \infty} \ef{\theta}{n} < \infty} \right]
        = \int_{\Theta_{\varepsilon/2}} \prior(\rmd \theta) \,
        \PP \left( \liminf_{n \to \infty} \ef{\theta}{n} < \infty \right) = 0 \eqsp.
    $$
    Hence, $\PP$-a.s.,
    $$
        \int_{\Theta_{\varepsilon/2}} \prior(\rmd \theta) \,
        \indiceTxt{\liminf_{n \to \infty} \ef{\theta}{n} < \infty} = 0 \eqsp,
    $$
    and, consequently, by Fatou's lemma and the fact that
    $\prior(\Theta_{\varepsilon/2}) > 0$,
    $$
        \liminf_{n \to \infty} \int_{\Theta_{\varepsilon/2}}
        \prior(\rmd \theta) \, \ef{\theta}{n}
        \geq \int_{\Theta_{\varepsilon/2}} \prior(\rmd \theta) \,
        \liminf_{n \to \infty} \ef{\theta}{n} \indiceTxt{\liminf_{n \to \infty} \ef{\theta}{n} = \infty}
        = \infty \eqsp,
    $$
    $\PP$-a.s., which completes the proof.
\end{proof}


Using the previous lemma, the proof of \Cref{thm:general:post:const}
is straightforward. Indeed, let $A \in \mct$ be approximately $\PP$-remote;
then for all
$\varepsilon > 0$ there exists $K_\varepsilon \in \mct$ such that
$\PP$-a.s.,
$$
    \limsup_{n \to \infty} \prior_n(A) \leq \limsup_{n \to \infty}
    \prior_n( A \cap K_\varepsilon)
    + \varepsilon \eqsp.
$$
To treat further the right hand side above, let
\begin{equation} \label{eq:bidule}
    \alpha \eqdef - \limsup_{n \to \infty} n^{-1} \log \left( 
    \int_{A \cap K_\varepsilon} \prior(\rmd \theta) \,
    \frac{\gd{\theta}{n}}{\gdstar{n}} \right) \eqsp,
\end{equation}
which is $\PP$-a.s. positive by \Cref{def:admissible}\ref{item:ass:Compact}, and write
\begin{equation} \label{eq:posterior:first:term:decomposition}
    \prior_n(A \cap K_\varepsilon) = \left( \e^{\alpha n / 2}
    \int_{A \cap K_\varepsilon} \prior(\rmd \theta) \,
    \frac{\gd{\theta}{n}}{\gdstar{n}} \right) \left( \e^{\alpha n / 2} \int \prior(\rmd \theta) \, \frac{\gd{\theta}{n}}{\gdstar{n}}
    \right)^{-1} \eqsp.
\end{equation}
Here, by \Cref{eq:lemma:merging:w:pr:one},
$$
    \limsup_{n \to \infty} \left( \e^{\alpha n / 2} \int \prior(\rmd \theta) \, \frac{\gd{\theta}{n}}{\gdstar{n}}
    \right)^{-1} \leq 1  \eqspeq\PP\mbox{-a.s.,}
$$
and applying this bound together with \eqref{eq:bidule} and
\eqref{eq:posterior:first:term:decomposition} yields, $\PP$-a.s.,
$$
    \limsup_{n \to \infty} \prior_n(A \cap K_\varepsilon) = 0 \eqsp.
$$
Consequently, $\limsup_n \prior_n(A) \leq \varepsilon$ $\PP$-a.s., and as $\varepsilon$ was picked arbitrarily
we may conclude the proof.


\section{Proof of main results}
\label{sec:proof-main-results}
\subsection{Preliminaries}

We use the general setting detailed in \Cref{sec:general-setting} for proving the results in \Cref{sec:main:results}. In \Cref{sec:exple-iid}, also the i.i.d. case is embedded into the general setting for completeness. Here the $\sigma$-finite measure $\nu_n$ is defined
similarly (see~\eqref{eq:nu-n-def-canonic}), but in the present case, for a given initial distribution
$\Xinit$, we define, for all $\theta\in\Theta$ and $n\in \nsetpos$, $\gd{\theta}{n}$
as the density of ${\PP^\theta_\Xinit}|_{\mcf_n}$ w.r.t. $\nu_n$, i.e., it
satisfies, for all $B={[\chunk Y1n]}^{-1}(A)\in\mcf_n$ with
$A\in\Ysigma^{\tensprod n}$,
\begin{equation}
  \label{eq:gd-def-pomm}
\int_{B} \gd{\theta}{n}\;\rmd\nu_n = \PP^\theta_\Xinit(B)=
\PP^\theta_\Xinit(\chunk Y1n\in A)\eqsp.
\end{equation}
This density is simply given by
$$
\gd{\theta}{n} = \py{\theta}{\eta}{\chunk Y1n} \eqsp.
$$
\begin{remark}\label{rem:merging}
In \Cref{sec:exple-iid} the true density is among the targeted ones,
$\gdstar{n}=\gd{\thv}{n}$. Although we here assume a \emph{true parameter}
$\thv\in\Theta$, we do not suppose that $\gdstar{n}=\gd{\thv}{n}$. The main reason is that in
the case of {\fdPOMM}s, the initial distribution $\Xinit$
in~\eqref{eq:gd-def-pomm} is chosen arbitrarily, often for computational
convenience.  More specifically, here $\gd{\thv}{n}=\py{\thv}{\Xinit}{\chunk Y1n}$, where the initial distribution
$\Xinit$ used in practice when computing the likelihood is chosen arbitrarily as one generally
 different from the \emph{true} initial distribution.
Concerning the latter we will instead consider $\gdstar{n}=\py{\thv}{\piv}{\chunk Y1n}$ (i.e. the true initial distribution is the
invariant one).
\end{remark}
Following this remark, the first thing to check is the merging
property~\eqref{eq:merge-at-thetastar} with
$\gdstar{n}=\py{\thv}{\piv}{\chunk Y1n}$ and
$\gd{\thv}{n}=\py{\thv}{\Xinit}{\chunk Y1n}$, i.e.,
\begin{equation}
  \label{eq:merging:pomm}
  \lim_{n \to \infty} \frac1n\log \frac{\py{\thv}{\Xinit}{\chunk Y1n}}
{\py{\thv}{\piv}{\chunk Y1n}}= 0 \eqspeq {\PP^{\thv}_{\piv}} \as
\end{equation}
As we will see below, this condition is implied by the following assumption.
\begin{hyp}{B}
\item\label{ass:pomm:equiv} For all $\theta\in\Theta$ and all initial
  distributions $\eta$, the distribution of $\chunk{Y}1{\infty}$ under
  $\PP_{\Xinit}^{\theta}$ admits a positive density with respect to the
  distribution of $\chunk{Y}1{\infty}$ under $\PP_{\pi_\theta}^{\theta}$, i.e.,
$$
R^\theta_{\Xinit} \eqdef \frac{\rmd \PP_{\Xinit}^{\theta}(\chunk{Y}1{\infty}\in\cdot)}
{\rmd \PP_{\pi_\theta}^{\theta}(\chunk{Y}1{\infty}\in\cdot)}>0
\quad {\PP^{\theta}_{\pi_\theta}}\mbox{-a.s.}
$$
\end{hyp}
We now state a result that will be shown to apply under the various sets of
assumptions in \Cref{sec:proof-crefthm:main} and \Cref{sec:proof-crefthm:hmm}
\begin{proposition}\label{prop:preliminaries}
  Consider a \fdPOMM\ satisfying \refhyp[B]{ass:stat:dist}[ass:amle:pomc] and
  \refhyp[B]{ass:pomm:equiv}.  Let $\lambda$ be a Radon measure on $\Theta$, 
  $\thv\in\Theta$ and define, for all $\chunk{y}{1}{n} \in \Yset^n$, $\post{\chunk{y}{1}{n}}$
  by~\eqref{eq:post-def-pomc}.  Let us
  consider the following conditions.
  \begin{enumerate}[label=(\roman*),resume=preliminaries-ass]
  \item  \label{ass:pomm-Keps} For all $\epsilon>0$, there exists a compact set $K_\epsilon \in \mct$ such that
$$\limsup_{n\to\infty}\post{\chunk{Y}{1}{n}}(K_\epsilon^c)\leq \epsilon\quad \PP^{\thv}_{\piv}\as  $$
  \item  \label{ass:pomm-general} Assumption~\refhyp[A]{ass:general} holds with
$\gd{\theta}{n}=\py{\theta}{\Xinit}{\chunk{Y}{1}{n}}$, $\gdstar{n} =
\py{\thv}{\piv}{\chunk{Y}{1}{n}}$ and $\PP=\PP^\thv_\piv$.
  \item  \label{ass:pomm-general-positive} For $n$ large enough, we have
$\py{\theta}{\Xinit}{\chunk{y}{1}{n}}>0$ for $\nu^{\tensprod n}\pp$
$\chunk y1n\in\Yset^n$ and $\prior\pp$ $\theta\in\Theta$.
\end{enumerate}
Then~\ref{ass:pomm-Keps} implies
  \begin{enumerate}[label=(\alph*),resume=preliminaries]
  \item  \label{assertion:bayesian-consistency:pomm-premote}  All closed sets
    $A \in \mct$ that do not contain $\thv$ are approximately $\PP^{\thv}_{\piv}$-remote.
  \end{enumerate}
and~\ref{ass:pomm-Keps}--\ref{ass:pomm-general-positive} imply
  \begin{enumerate}[label=(\alph*),resume=preliminaries]
\item   \label{assertion:bayesian-consistency:pomm} For all initial
  distributions $\Xinitv$,
$\displaystyle\post{\chunk{Y}{1}{n}}  \weakconv{n \to \infty} \delta_\thv$, $\PP^\thv_\Xinitv\as$
\end{enumerate}
\end{proposition}

\begin{proof}
If~\eqref{eq:merging:pomm} holds, then, setting $\PP=\PP^{\thv}_{\piv}$, \Cref{prop:equiv:AMLE}
and~\refhyp[B]{ass:amle:pomc} give immediately that $A\cap K$ is $\PP$-remote for
all closed sets $A$ not containing $\thv$.
Hence, in order to establish~\ref{assertion:bayesian-consistency:pomm-premote}, we only need to check that~\refhyp[B]{ass:pomm:equiv}
implies~\eqref{eq:merging:pomm}. For this purpose, assume~\refhyp[B]{ass:pomm:equiv}
and set $R^*\eqdef R^\thv_{\Xinit}$, which then is $\PP\as$ positive 
Then $\PE^\thv_\piv[R^*]=1$ and
\begin{equation}
  \label{eq:pomm:equiv:n-finite}
  \frac{\py{\thv}{\eta}{\chunk Y1n}}{\py{\thv}{\piv}{\chunk Y1n}}=\PE^\thv_\piv\left[ R^* \mid \chunk Y1n\right]\quad \PP\as,
\end{equation}
from which we conclude that the left hand side converges to $R^*>0$ $\PP\as$
and in $L^1$ (see, e.g., \cite[Theorem~27.3]{jacod:protter:2000}). This
implies~\eqref{eq:merging:pomm}
and~\ref{assertion:bayesian-consistency:pomm-premote} follows.

We now assume, additionally,~\ref{ass:pomm-general} and~\ref{ass:pomm-general-positive}. Then, by
\Cref{thm:general:post:const} and \Cref{rem:post:const}, it suffices, in order to
obtain~\eqref{eq:weak-conv-bayesan-consistency}, to check
that $\lambda_n=\post{\chunk{Y}{1}{n}}$ is, $\PP\as$, a well-defined probability for $n$ large
enough. By~\ref{ass:pomm-Keps} there is a compact set $K$ such that
$\int_{K^c} \lambda(\rmd \theta) \, \gd{\theta}{n} <\infty$ for $n$ large enough
$\PP\as$, and since $\lambda$ is a Radon measure, it also holds that $\int_K
\lambda(\rmd \theta) \, \gd{\theta}{n} <\infty$ $\nu_n$-almost everywhere and thus
$\PP\as$ Hence $\int_{\Theta} \lambda(\rmd \theta) \, \gd{\theta}{n} <\infty$ for
$n$ large enough, $\PP\as$ By \Cref{rem:true-posterior-consistency}, it only
remains to check that $\int_{\Theta} \lambda(\rmd \theta) \, \gd{\theta}{n}>0$
$\PP\as$, which is directly implied by~\ref{ass:pomm-general-positive}.

Hence, we get that
$\post{\chunk{Y}{1}{n}}$ is a well-defined probability for $n$ large enough
$\PP\as$ This yields~\eqref{eq:weak-conv-bayesan-consistency}, which corresponds to
Assertion~\ref{assertion:bayesian-consistency:pomm} in the special case
$\Xinitv=\piv$. Under~\refhyp[B]{ass:pomm:equiv}, this also implies
Assertion~\ref{assertion:bayesian-consistency:pomm} for all initial
distributions  $\Xinitv$.
\end{proof}

\subsection{Proof of \Cref{thm:main:result}}
\label{sec:proof-crefthm:main}
Before proving \Cref{thm:main:result}, we need two preliminary results. The
first is a change of probability formula (see \Cref{lem:radon} below), which
will be used for extending the posterior consistency property to a non-stationary
sequence. The second result allows \refhyp[A]{ass:general} to be checked in order to apply
\Cref{thm:general:post:const} (see \Cref{lem:check:A1} below).

\begin{lemma} \label{lem:radon}
If a \fdPOMM\ satisfies \refhyp[B]{ass:stat:dist} and
\refhyp[B]{ass:q:positive}, then it also satisfies~\refhyp[B]{ass:pomm:equiv}.
\end{lemma}

\begin{proof}
  Let $\theta\in\Theta$ and $\Xinit$ be an initial distribution on $(\Zset,
  \Zsigma)$.  Under $\PP^{\theta}_{\Xinit}$, $\sequence{Z}[k][\nsetpos]$ is a
  Markov chain with transition kernel $\Tk[\theta]$ and initial distribution
  having the density
$$
z_1 \mapsto \Xinit_1(z_1) =\int \Xinit(\rmd z) \, \Td[\theta](z,z_1)
$$
\wrt\  $\mu \tensprod \nu$. Similarly, under $\PP^{\theta}_{\pi_\theta}$, $\sequence{Z}[k][\nsetpos]$ is a Markov chain with the same transition kernel $\Tk[\theta]$ as above but with another initial distribution having the density
$$
z_1 \mapsto \pi_\theta(z_1) =\int \pi_\theta(\rmd z) \, \Td[\theta](z,z_1)
$$
\wrt\  $\mu \tensprod \nu$.

Under \refhyp[B]{ass:q:positive}, these two densities are positive and
$\chunk{z}{1}{\infty} \mapsto \eta_1(z_1)/\pi_\theta(z_1)$ is thus the Radon-Nikodym
ratio between $\PP^{\theta}_{\Xinit}$ and $\PP^{\theta}_{\pi_\theta}$ restricted to
$\sigma(\chunk Z1{\infty})$. The result follows.
\end{proof}

\begin{lemma}\label{lem:check:A1}
  Consider a\ \fdPOMM\ satisfying~\refhyp[B]{ass:stat:dist}. Then
  Assumptions~\refhyp[B]{ass:q:positive} and \refhyp[B]{ass:kullback} imply
  \refhyp[A]{ass:general} with
  $\gd{\theta}{n}=\py{\theta}{\Xinit}{\chunk{Y}{1}{n}}$, $\gdstar{n} =
  \py{\thv}{\piv}{\chunk{Y}{1}{n}}$, $\PP=\PP^\thv_\piv$ and
  $\Theta_\delta=\ens{\theta \in \Theta}{\D(\thv, \theta)\leq\delta}$.
\end{lemma}

\begin{proof}
Pick $\delta>0$ and take any $\theta\in\Theta$ such that $\D(\thv,
  \theta)\leq\delta$. To prove the result, we need to show that
  \begin{equation}
    \label{eq:fully-dom-kullback-toprove}
\liminf_{n \to \infty} n^{-1} \log \frac{\gd{\theta}{n}}{\gdstar{n}}
\geq - \delta  \quad \PP\as\eqsp,
  \end{equation}
where $\gd{\theta}{n}=\py{\theta}{\Xinit}{\chunk{Y}{1}{n}}$ and $\gdstar{n}=\py{\thv}{\piv}{\chunk{Y}{1}{n}}$.
We apply \Cref{lem:image-density-lemma} with $\Omega=\Zset^\nset$,
$\mcf=\Zsigma^{\tensprod\nset}$, $\overline{\mcf}_n=\sigma(\chunk Z1n)$,
$\PP=\PP^{\thv}_{\piv}$ and $\overline{\nu}_n$ given (as
in~\eqref{eq:nu-n-def-canonic}) by
$$
\overline\nu_n(B) = (\mu \tensprod \nu)^{\tensprod n}(A) \eqsp, \ B\in\overline\mcf_n
\text{ with $B={[\chunk Z1n]}^{-1}(A)$ and $A\in\Zsigma^{\tensprod n}$}\eqsp.
$$
Then $\bgdn{n}$ and $\bgdstar{n}$  are the corresponding densities
$$
\bgdn{n}=\int \Td[\theta](z_0, Z_1) \, \eta(\rmd z_0) \, \prod_{k=1}^{n-1}
\Td[\theta](Z_k, Z_{k+1})\eqsp,
$$
and
$$
\bgdstar{n}=\piv(Z_1) \prod_{k=1}^{n-1}
\Td[\thv](Z_k,Z_{k+1})\eqsp.
$$
Moreover, set $\tilde\mcf_n=\sigma(\chunk Y 1n)$, so that $\nu_n$ is the
restriction of $\overline{\nu}_n$ to $\tilde\mcf_n$. In this
case that the densities introduced in~\eqref{eq:ratio-image-lemma-new-density}
are $\tgdn{n}=\gd{\theta}{n}$ and $\tgdstar{n}=\gdstar{n}$. Thus, applying
\Cref{lem:image-density-lemma}, we get
that~\eqref{eq:fully-dom-kullback-toprove} is implied by
\begin{equation}
  \label{eq:to-show-fully-dominatedA1}
\liminf_{n\to\infty}  \frac1n\log  \frac{\bgdn{n}}{\bgdstar{n}} \geq
-\delta \quad \PP^\thv_{\piv}\-\as
\end{equation}
Now, observe that, $\PP^\thv_\piv\as$,
$$
\log \frac{\bgdn{n}}{\bgdstar{n}}
=
\log\frac
{\int \Td[\theta](z_0,Z_1)\;\eta(\rmd z_0)}
{\piv(Z_1)}
+\sum_{\ell = 1}^{n-1} \log\frac
{\Td[\theta](Z_\ell,Z_{\ell+1})}
{\Td[\thv](Z_\ell,Z_{\ell+1})} \eqsp.
$$
Since the transition density $\Td[\theta]$ is assumed to be positive, the first
term is a finite number and tends to zero when divided by $n$, $\PP^\thv_{\piv}\-\as$
Moreover, by~{\refhyp[B]{ass:stat:dist}}\,, $Z$ is ergodic under $\PP^\thv_{\piv}$, and we obtain
$$
\liminf_{n\to\infty}n^{-1}\sum_{\ell = 1}^{n-1} \log\frac
{\Td[\theta](Z_\ell,Z_{\ell+1})}
{\Td[\thv](Z_\ell,Z_{\ell+1})} =-\D(\thv,\theta) \quad \PP^\thv_{\piv}\-\as
$$
Since we have assumed that $\D(\thv,\theta)\leq\delta$, \eqref{eq:to-show-fully-dominatedA1} holds true and the proof is completed.
\end{proof}

The proof of \Cref{thm:main:result} is now completed by observing that
\Cref{lem:radon,lem:check:A1} show that the assumptions of
\Cref{thm:main:result} imply those of \Cref{prop:preliminaries}.

Note that~\ref{ass:pomm-Keps} in \Cref{prop:preliminaries} is trivially
satisfied when $\Theta$ is compact, and that~\ref{ass:pomm-general-positive}
directly follows from ~\refhyp[B]{ass:q:positive}.

\subsection{Proof of \Cref{thm:hmm}}
\label{sec:proof-crefthm:hmm}
The only modification of the proof consists in observing that
Condition~\ref{ass:pomm-general-positive} in \Cref{prop:preliminaries}
now directly follows
from~\refhyp[C]{ass:fully:dominated:restricted}[ass:g:positive]
 and in showing that the conclusions of
\Cref{lem:radon} and \Cref{lem:check:A1} hold true under the new set of
assumptions. This is done in \Cref{lem:radon:hmm} and
\Cref{lem:check:A1:hmm} below.

\begin{lemma} \label{lem:radon:hmm} If a \fdPOMM\
  satisfies~\refhyp[B]{ass:stat:dist},~\refhyp[C]{ass:fully:dominated:restricted},~\refhyp[C]{ass:g:positive}
  and~\refhyp[C]{ass:TVforgetting}, then it also
  satisfies~\refhyp[B]{ass:pomm:equiv}.
\end{lemma}

\begin{proof}
  Let $\theta\in\Theta$ and $\Xinit$ be an initial distribution on
  $(\Zset, \Zsigma)$. Then under
  Assumptions~\refhyp[B]{ass:stat:dist},~\refhyp[C]{ass:fully:dominated:restricted}
  and~\refhyp[C]{ass:TVforgetting} there exists a sequence
  $\sequencen{(X'_k,X''_k,Y_k)}[k\in\nset]$ such that
  \begin{enumerate}
  \item   $\sequencen{(X'_k,Y_k)}[k\in\nset]$ is distributed according to
  $\PP_{\pi_\theta}^\theta$,
\item    $\sequencen{(X''_k,Y_k)}[k\in\nset]$ is distributed according to
  $\PP_{\Xinit}^\theta$,
\item there is a ${\PP^\theta_{\pi_\theta}}\as$ finite stopping time $\tau$
  such that $X'_k=X''_k$ for all $k>\tau$.
  \end{enumerate}
  See \cite[Lemma~3.7]{vanhandel:2009} and also the end of
  the proof of \Cref{lem:check:A1:hmm} where the same construction is used.
Then, using~\refhyp[C]{ass:g:positive} and mimicking
\cite[Lemma~3.7]{vanhandel:2009} yields that the laws
of $\sequencen{(X'_k,Y_k)}[k\in\nset]$ and
$\sequencen{(X''_k,Y_k)}[k\in\nset]$ are equivalent, which in its turn imply~\refhyp[B]{ass:pomm:equiv}.
\end{proof}

\begin{lemma}\label{lem:check:A1:hmm}
  Consider a {\fdPOMM}\ with initial
  distribution $\Xinit$ on $(\Zset,\Zsigma)$.  Assume \refhyp[B]{ass:stat:dist}
  and~\refhyp[C]{ass:fully:dominated:restricted} and set
  $\PP=\PP^{\thv}_{\piv}$. Then \refhyp[C]{ass:g:positive}[ass:TVforgetting] imply \refhyp[A]{ass:general} with $\gd{\theta}{n}=\py{\theta}{\Xinit}{\chunk{Y}{1}{n}}$, $\gdstar{n} =
\py{\thv}{\piv}{\chunk{Y}{1}{n}}$, $\PP=\PP^\thv_\piv$ and 
$\Theta_\delta=\ens{\theta\in \Theta}{\Dhmm(\thv, \theta)\leq\delta}$.
\end{lemma}
\begin{proof}
Pick $\delta>0$ and take any $\theta\in\Theta$ such that $\Dhmm(\thv,
  \theta)\leq\delta$. By~\refhyp[C]{ass:kullback:hmm} it is sufficient to show that
  \begin{equation}
\liminf_{n \to \infty} n^{-1} \log \frac{\gd{\theta}{n}}{\gdstar{n}}
\geq - \delta  \quad \PP\as\eqsp,
  \end{equation}
  where $\gd{\theta}{n}= \py{\theta}{\Xinit}{\chunk Y1n}$ and $\gdstar{n}= \py{\thv}{\piv}{\chunk Y1n}$.  The idea of the proof is now similar to
  that of \Cref{lem:check:A1}, but with a completely different augmented set of
  variables. Instead of augmenting the data by just the unobserved sequence $(X_k)$, we
  now add one more sequence $(X'_k)$ to the data as follows.  Let $\Zset'=\Xset^2\times\Yset$ and $\Zsigma'=\Xsigma^{\tensprod2}\tensprod\Ysigma$ and denote, for all
  $k\in\nset$, by $Z_k=(X_k,Y_k)$ and $Z'_k=(X_k,X'_k,Y_k)$ the members of the corresponding canonical sequences. We define $\PP$ as the distribution of $\sequence{Z'}[n][\nset]$
  when $\sequence{Z}[n][\nset]$ is distributed according to $\PP^{\thv}_{\piv}$,
  $\sequence{X'}[n][\nset]$ is the canonical Markov chain with initial distribution
  $\Xinit$ and kernel $\Tx[\theta]$ and, moreover, $\sequence{X'}[n][\nset]$ and
  $\sequence{Z}[n][\nset]$ are independent.

We apply \Cref{lem:image-density-lemma} with $\Omega=\Zset^{\prime\nset}$,
$\mcf=\Zsigma^{\prime\tensprod\nset}$, $\overline{\mcf}_n=\sigma(Z'_{1:n})$ and $\overline{\nu}_n$ given (as
in~\eqref{eq:nu-n-def-canonic}) by
$$
\overline\nu_n(B)=(\mu^{\tensprod2}\tensprod\nu)^{\tensprod n}(A)\eqsp, \ B\in\overline\mcf_n
\text{ with $B={[Z'_{1:n}]}^{-1}(A)$ and $A\in\Zsigma^{\prime\tensprod n}$}\eqsp.
$$
In this particular setting, $\bgdstar{n}$ takes the form
$$
\bgdstar{n}=\piv(Z_1)\int\Tx[\theta](x'_0,X'_1) \, \Xinit(\rmd x'_0) \prod_{k=1}^{n-1}
\Td[\thv](Z_k,Z_{k+1})\Tdx[\theta](X'_k,X'_{k+1}) \eqsp.
$$
Now, define $\PP_{\theta}$ as the distribution of  $\sequence{Z'}[n][\nset]$
  when $\sequence{X}[n][\nset]$ and $\sequence{X'}[n][\nset]$ are distributed exactly as under $\PP$
  (i.e., two independent Markov chains with kernels  $\Tx[\thv]$ and $\Tx[\theta]$
  and initial distributions $\piv$ and $\Xinit$, respectively), but, conditionally
  on these sequences, $\sequence{Y}[n][\nset]$ has the law of a sequence
  of independent random variables such that for all $k$, $Y_k$ has density
$\Tdy[\theta](X'_k,\cdot)$ with respect to $\nu$. Hence, $\sequencen{(X'_k,Y_k)}[k\in\nset]$ is an HMM
with parameter $\theta$ and initial distribution $\Xinit$.
Recall that we define  $\bgdn{n}$ as the density of the distribution
$\PP_{\theta,n}$, which in its turn is defined as the restriction of $\PP_{\theta}$ on
$\overline{\mcf}_n$. Consequently,
\begin{multline*}
\bgdn{n}=\piv(X_1,\Yset) \int\Tx[\theta](x'_0,X'_1) \, \Xinit(\rmd x'_0)
\\\times \left( \prod_{k=1}^{n-1}
\Tdx[\thv](X_k,X_{k+1})\Tdx[\theta](X'_k,X'_{k+1})\Tdy[\theta](X'_k,Y_k) \right) \Tdy[\theta](X'_n,Y_n)\eqsp,
\end{multline*}
where $\piv(\cdot,\Yset)$ denotes the marginal of the density $\piv$ w.r.t. the $\Xset$ component.
It now holds that
$$
\frac{\bgdn{n}}{\bgdstar{n}}=
 \prod_{k=1}^{n}\frac{\Tdy[\theta](X'_k,Y_k)}{\Tdy[\thv](X_k,Y_k)}\eqsp.
$$
Now, set $\tilde{\mcf}_n=\sigma(\chunk Y1n)$; then, defining
$\tgdn{n}$ and $\tgdstar{n}$ as in~\eqref{eq:ratio-image-lemma-new-density} yields straightforwardly that $\tgdn{n}=\gd{\theta}{n}$ and $\tgdstar{n}=\gdstar{n}$.
Hence by \Cref{lem:image-density-lemma},
Eqn.~\eqref{eq:fully-dom-kullback-toprove} follows from
  \begin{equation}
    \label{eq:fully-dom-kullback-toprove-now}
\liminf_{n \to \infty} n^{-1} \log \frac{\bgdn{n}}{\bgdstar{n}}
\geq - \delta  \quad \PP\as\eqsp,
  \end{equation}
whose establishment is the object of the remainder of the proof.

We prove \eqref{eq:fully-dom-kullback-toprove-now} by means of a coupling
argument used in \cite[Lemma~3.7]{vanhandel:2009}.  Recall that under $\PP$,
$\sequence{Z}[n][\nset]$ and $\sequence{X'}[n][\nset]$ are independent. Thus,
by Condition~\refhyp[C]{ass:TVforgetting}\,, following \cite[Theorem
III.14.10]{lindvall:1992}, we extend $\PP$ by adding an $\Xset$-valued process
$\sequence{X''}[n][\nset]$ to $\sequence{Z'}[n][\nset]$ such that
$\sequence{X''}[n][\nset]$ is independent of $\sequence{Z}[n][\nset]$, has
distribution $\PP^\theta_{\pi_\theta}$ and
$$
\tau:=\min\{k \in \nset : X''_\ell=X'_\ell\text{ for all }\ell\geq k\}<\infty\quad \PP\as
$$
Then by~\refhyp[C]{ass:g:positive} we have, for all $n\geq\tau$,
$$
\frac{\bgdn{n}}{\bgdstar{n}}=
\left[ \prod_{k=1}^{\tau}\frac{\Tdy[\theta](X'_k,Y_k)}{\Tdy[\thv](X_k,Y_k)} \prod_{k=1}^{\tau}\frac{\Tdy[\thv](X_k,Y_k)}{\Tdy[\theta](X''_k,Y_k)}\right]
 \prod_{k=1}^{n}\frac{\Tdy[\theta](X''_k,Y_k)}{\Tdy[\thv](X_k,Y_k)}\eqsp,
$$
where the term within the brackets is positive $\PP\as$
Now, by \Cref{lem:ergodic:indep:product}, $\sequencen{(X_k,Y_k,X''_k)}[k\in\nset]$ is a
stationary ergodic Markov chain under $\PP$. Hence,
$$
\lim_{n \to \infty} n^{-1} \log \frac{\bgdn{n}}{\bgdstar{n}}=
\mathbb{E} \left[ \log \frac{\Tdy[\theta](X''_0,Y_0)}{\Tdy[\thv](X_0,Y_0)} \right] \quad\PP\as,
$$
where $\mathbb{E}$ denotes the expectation under $\PP$. 
To conclude, we observe that the latter expectation is exactly minus
$\Dhmm(\thv,\theta)$ as defined in~\eqref{eq:def-dhmm}, and thus the choice of
$\theta$ at the beginning of this proof
gives~\eqref{eq:fully-dom-kullback-toprove-now}.
\end{proof}

\subsection{Proof of \Cref{thm:noncompact}}
\label{sec:proof-crefthm:tightness}

Since Condition~\ref{ass:pomm-Keps} in \Cref{prop:preliminaries} is trivially
satisfied in the case where $\Theta$ is assumed to be compact, the proofs of
\Cref{thm:main:result} and \Cref{thm:hmm} only
required~\refhyp[B]{ass:pomm:equiv} and Condition~\ref{ass:pomm-general} of
\Cref{prop:preliminaries} to be checked.  The latter two assumptions are still
implied by those of \Cref{thm:noncompact}; however, since $\Theta$ is no longer
assumed to be compact, it remains to show that Condition~\ref{ass:pomm-Keps} in
\Cref{prop:preliminaries} still holds under the assumptions of
\Cref{thm:noncompact}. This will be done in \Cref{prop:hors:compact} below.

We preface this result with two lemmas.
  \begin{lemma}
    \label{lemma:noncompactcond-n0-all-n}
Consider a \fdPOMM\ satisfying \refhyp[B]{ass:stat:dist} and let $\lambda$ be a
positive measure on $(\Theta, \mct)$.
Then each condition in \refhyp[B]{ass:additive} implies the same condition with $n_0$ replaced
  by any $n\geq n_0$.
\end{lemma}
  \begin{proof} 
Let $n> n_0$. Using that for all $\chunk{y}{1}{n_0} \in \Yset^{n_0}$, 
\begin{equation}
  \label{eq:trivial-relation-n-n0}
  \iint
  \prior(\rmd \theta) \,
  \py{\theta}{\Xinit}{\chunk{y}{1}{n}} 
  \, \nu^{\tensprod (n-n_0)}(\rmd \chunk{y}{n_0 + 1}{n})
  =
\int \prior(\rmd \theta) \, \py{\theta}{\Xinit}{\chunk{y}{1}{n_0}},
\end{equation}
and that, under $\PP^\thv_\piv$,
$\py{\thv}{\piv}{\chunk
  y{n_0+1}n \mid \chunk{Y}{1}{n_0}} = \py{\thv}{\piv}{\chunk{Y}{1}{n_0},\chunk
  y{n_0+1}n}/\py{\thv}{\piv}{\chunk{Y}{1}{n_0}}$ is the density, with
respect to $\nu^{\tensprod (n-n_0)}$, of the
conditional distribution of $\chunk Y{n_0+1}n$ given $\chunk{Y}{1}{n_0}$, we get that
$$
\cesp[\piv]{\frac{\int \prior(\rmd \theta) \,
    \py{\theta}{\Xinit}{\chunk{Y}{1}{n}}}{
\py{\thv}{\piv}{\chunk
  Y{n_0+1}n\mid \chunk{Y}{1}{n_0}}}}{\chunk{Y}{1}{n_0}}[\thv]=
\int \prior(\rmd \theta) \, \py{\theta}{\Xinit}{\chunk{Y}{1}{n_0}}.
$$
Thus, by \Cref{lem:zarbi}, if~\eqref{eq:B6-1} holds, it also holds with $n$
replacing $n_0$.

Now, concerning~\eqref{eq:B6-2}, the comments before \cite[Theorem
1]{barron:1985} show that, using that the observations are stationary under $\PP^\thv_\piv$,
$(\PE^\thv_\piv[\log \py{\thv}{\piv}{Y_{n+1} \mid \chunk{Y}{1}{n}}] )_{n \in \nsetpos}$ is a nondecreasing sequence.
Hence if~\eqref{eq:B6-2} holds, it continues to hold true when $n_0$ is replaced by any
   $n\geq n_0$.
\end{proof}

\begin{lemma} \label{lem:tight:tech}
Consider a \fdPOMM\ satisfying \refhyp[B]{ass:stat:dist} and
  \refhyp[B]{ass:tight:bound}. Then for all $\delta>0$ there exists a compact
  $K\in\mct$ such that
$$
\limsup_{n \to \infty} n^{-1} \log \psup[K^c]{\chunk{Y}{0}{n}} \leq -\delta \quad \PP^\thv_\piv\as,
$$
where $\psup[K^c]{\chunk{Y}{0}{n}}$ is defined in~\eqref{eq:psup-def}.
\end{lemma}

\begin{proof}
We first show that
\begin{equation}
\label{eq:tightness:one}
\limsup_{m \to \infty} \PE^\thv_\piv[\log \psup[C_m^c]{\chunk{Y}{0}{\ell}}]=-\infty\eqsp.
\end{equation}
Set $U \eqdef \log^+\psup[\Theta]{\chunk{Y}{0}{\ell}}$ and $U_m \eqdef \log
\psup[C_m^c]{\chunk{Y}{0}{\ell}}$. First note that $\PE^\thv_\piv[U_m]$ is
well-defined since by~\eqref{eq:log:moment} in\ \refhyp[B]{ass:tight:bound}, $\PE^\thv_\piv[U_m^+] \leq
\PE^\thv_\piv[U]<\infty$. Now, since $U-U_m\geq 0$, Fatou's lemma yields
\begin{multline*}
\liminf_{m \to \infty}\PE^\thv_\piv[U-U_m]=\PE^\thv_\piv[U]-\limsup_{m \to \infty}\PE^\thv_\piv[U_m] \\
\geq \PE^\thv_\piv\left[\liminf_{m \to \infty}(U-U_m)\right]=\PE^\thv_\piv[U]-\PE^\thv_\piv\left[\limsup_{m \to \infty} U_m\right]\eqsp.
\end{multline*}
Combining with \eqref{eq:conv:zero}  in \refhyp[B]{ass:tight:bound} yields
$$
\limsup_{m \to \infty} \PE^\thv_\piv[\log \psup[C_m^c]{\chunk{Y}{0}{\ell}}] \leq \PE^\thv_\piv\left[\limsup_{m \to \infty}  \log \psup[C_m^c]{\chunk{Y}{0}{\ell}}\right]=-\infty\eqsp,
$$
and \eqref{eq:tightness:one} is shown. Now, let $\delta>0$. According to \eqref{eq:tightness:one}, one may pick $m\in \nset$ sufficiently large such that
$$
\ell^{-1}\PE^\thv_\piv[\log \psup[C_m^c]{\chunk{Y}{0}{\ell}}] \leq -\delta\eqsp.
$$
Now, set $K\eqdef C_m$ and define, for $(r, s) \in \nset^2$ such that $r \leq s$, $W_{r,s} \eqdef \psup[K^c]{\chunk{Y}{r}{s}}$. By
(\ref{eq:log:moment}) in \refhyp[B]{ass:tight:bound},
$\PE[\log^+W_{0,\ell}]<\infty$ and for all $r \leq s \leq t$,
$$
W_{r,t} \leq W_{r,s}W_{s,t}\eqsp.
$$
Since under $\PP^\thv_\piv$, the sequence $\sequence{Y}[n][\nset]$ is
stationary and ergodic, the Kingman subadditive theorem (\cite{kingman:1973})
applies. Thus, $\lim_{n \to \infty} n^{-1}\log W_{0,n}$ exists
$\PP^\thv_\piv\as$ and
$$
\lim_{n \to \infty} n^{-1}\log W_{0,n}=\inf_{n \geq \ell} n^{-1}\PE^\thv_\piv\log W_{0,n}\leq \ell^{-1}\PE^\thv_\piv[\log \psup[K^c]{\chunk{Y}{0}{\ell}}] \leq -\delta\quad \PP^\thv_\piv\as
$$
The proof is completed.
\end{proof}

\begin{proposition} \label{prop:hors:compact} Consider a \fdPOMM\ satisfying
  \refhyp[B]{ass:stat:dist}, \refhyp[B]{ass:tight:bound} and
  \refhyp[B]{ass:additive}.  Assume \refhyp[A]{ass:general} with
  $\gd{\theta}{n}=\py{\theta}{\Xinit}{\chunk{Y}{1}{n}}$, $\gdstar{n} =
  \py{\thv}{\piv}{\chunk{Y}{1}{n}}$ and $\PP=\PP^\thv_\piv$. Then there exists
  a compact set $K \in \mct$ such that
    $$
        \limsup_{n \to \infty} n^{-1}\log \post{\chunk{Y}{1}{n}}(K^c)
        < 0  \quad  \PP^\thv_\piv\as
    $$
\end{proposition}

\begin{proof}
Following \cite{barron:1985}, let us define, for all $n \in \nsetpos$,
$$
 V_n\eqdef \py{\thv}{\piv}{\chunk{Y}{1}{n}}\eqsp,\quad \tilde V_n\eqdef \PE^\thv_\piv[\log(V_n/V_{n-1})].
$$
By~\eqref{eq:B6-2} in \refhyp[B]{ass:additive}, we have $\tilde V_{n_0}>-\infty$.
As explained in the comments before \cite[Theorem 1]{barron:1985},
$\sequence{\tilde V}[n][\nsetpos]$ is a non-decreasing sequence, and denoting by
$\tilde V$ its limit in $(-\infty,\infty]$, by
\cite[Theorem 1]{barron:1985}, we have
\begin{equation}
\label{eq:tight:proof:three}
\lim_{n \to \infty} n^{-1} \log V_n= \tilde V\quad \PP^\thv_\piv\as
\end{equation}
Pick $\delta>0$ and $\epsilon>0$ such that
$$
-\delta-\tilde V+\epsilon<0\eqsp.
$$
According to \Cref{lem:tight:tech}, there exists a compact set $K \in \mct$ such that
\begin{equation}
\label{eq:tight:proof:one}
\limsup_{n \to \infty} n^{-1} \log \psup[K^c]{\chunk{Y}{0}{n}} \leq -\delta \quad \PP^\thv_\piv\as
\end{equation}
Now, write for all $n \in \nset$ strictly larger than $n_0$, 
\begin{align}
\post{\chunk{Y}{1}{n}}(K^c) &\leq \left(\int_{K^c} \py{\theta}{\Xinit}{\chunk{Y}{1}{n_0}} \prior(\rmd \theta) \right)\frac{ \psup[K^c]{\chunk{Y}{n_0}{n}}}{ \py{\prior}{\Xinit}{\chunk{Y}{1}{n}}} \leq \py{\prior}{\Xinit}{\chunk{Y}{1}{n_0}} \frac{U_n}{V_n W_n},\label{eq:tight:proof:main}
\end{align}
where
\begin{align*}
& U_n \eqdef \psup[K^c]{\chunk{Y}{n_0}{n}}\eqsp,\quad W_n\eqdef \py{\prior}{\Xinit}{\chunk{Y}{1}{n}}/\py{\thv}{\piv}{\chunk{Y}{1}{n}}\eqsp.
\end{align*}
By the first condition in  \refhyp[B]{ass:additive},
\begin{equation}\label{eq:tight:proof:two}
\limsup_{n \to \infty}n^{-1}\log \py{\prior}{\Xinit}{\chunk{Y}{1}{n_0}}=0 \quad \PP^\thv_\piv\as
\end{equation}
Moreover, according to \Cref{eq:lemma:merging:w:pr:one}, \refhyp[A]{ass:general} implies
\begin{equation}\label{eq:tight:proof:four}
\liminf_{n \to \infty} n^{-1} \log W_n \geq -\epsilon \quad \PP^\thv_\piv\as
\end{equation}
Finally, combining \eqref{eq:tight:proof:three}, \eqref{eq:tight:proof:one}, \eqref{eq:tight:proof:two}, \eqref{eq:tight:proof:four} with \eqref{eq:tight:proof:main} yields:
$$
\limsup_{n \to \infty} n^{-1}\log\post{\chunk{Y}{1}{n}}(K^c)\leq
-\delta-\tilde V+\epsilon<0 \quad \PP^\thv_\piv\as
$$
\end{proof}

\section{Conclusion}
\label{sec:conclusion}
We have established that the posterior consistency
for {\fdPOMM}s is a consequence of the consistency of the AMLE under---what we
believe---minimal assumptions that can be checked for a variety of models used
in practice. Importantly, our assumptions can be checked for models where the
state space of the latent process is non-compact, which is most often the case
in applications (including, e.g., the linear Gaussian state-space
models). Moreover, we allow also the parameter space to be non-compact, which
is essential in the Bayesian setting (where many prior distributions of
fundamental importance have infinite support), and the prior to be
improper. Thus, our results generalize substantially existing results in this
direction, which focus exclusively on the special case of HMMs and require both
the state space of the hidden chain and the parameter space to be compact. Our
proofs rely on a machinery revolving around the general concept of
$\PP$-remoteness introduced in Section~\ref{sec:gener-approach}, which is
naturally linked to the consistency of the MLE in the frequentist setting. As
far as known to the authors, this link has not been explored before in the
literature.

Our analysis relies substantially on the assumption that the model is fully
dominated, which is certainly a restriction. A natural direction for research
is hence the relaxation of this assumption, and incorporating techniques
developed in \cite{douc-roueff-sim2015b} into the analysis could allow our
results to be extended to observation-driven models (including the GARCH
framework). Moreover, as we do not provide any rate of convergence,
supplementing our results with a Bernstein-von~Mises-type theorem would of
course be desirable; nevertheless, establishing such a theorem involves
typically, \emph{inter alia}, a law of large numbers for the observed Fisher
information, which appears to be a real challenge in the non-compact setting
(the proof of a Bernstein-von~Mises-type theorem is hence expected to be at
least on the same level of difficulty as the proof of the asymptotic normality
of the MLE, which is still an open problem in the non-compact case). Finally,
another natural topic for future research is the extension of our results in
the direction of nonparametric Bayesian modeling.


\appendix

\section{Some kernel notation}
\label{sec:kernel:notation}

Let $\mu$ be a signed measure on some measurable space $(\Xset, \Xsigma)$. For any $|\mu|$-integrable function $h$, we denote by
$$
\mu h \eqdef \int h(x) \, \mu(\rmd x)
$$
the Lebesgue integral of $h$ w.r.t. $\mu$. 

In addition, let $(\Yset, \Ysigma)$ be some other measurable space and $\kernel$ some possibly unnormalized transition kernel $\kernel : \Xset \times \Ysigma \rightarrow \rset_+$. The kernel $\kernel$ induces two integral operators, one acting on functions and the other on measures. More specifically, given a measure $\nu$ on $(\Xset, \Xsigma)$ and a measurable function $h$ on $(\Yset, \Ysigma)$, we define the measure 
$$
\nu \kernel : \Ysigma \ni A \mapsto \int \kernel(x, A) \, \nu(\rmd x) 
$$
and the function 
$$
\kernel h : \Xset \ni x \mapsto \int h(y) \, \kernel(x, \rmd y)
$$
whenever these quantities are well-defined. For the latter, we will, whenever convenient, use also the alternative notation $\kernel(\cdot, h)$. 

Finally, given a third measurable space $(\Zset, \Zsigma)$ and another kernel $\mathbf{L} : \Yset \times \Zsigma \rightarrow \rset_+$ we define, with $\kernel$ as above, the \emph{product kernel}
$$
\kernel \mathbf{L} : \Xset \times \Zsigma \ni (x, B) \mapsto \int  \mathbf{L}(y, B) \, \kernel(x, \rmd y),
$$
whenever this is well-defined. When $\kernel$ describes transitions within the \emph{same} space $(\Xset, \Xsigma)$, its \emph{iterates} are defined inductively by 
$$
\kernel^0(x, \cdot)  \eqdef \delta_x \mbox{ for all } x \in \Xset \mbox{ and } \kernel^n \eqdef \kernel^{n - 1} \kernel \mbox{ for all }n \in \nsetpos. 
$$

\section{Dominated i.i.d. model}
  \label{sec:exple-iid}
  The general setting in~\ref{sec:general-setting} is a bit unusual in the sense that all
  densities and dominating measures are defined directly on the same space
  $\Omega$ (endowed however with different $\sigma$-fields picked among the
  members of the sequence $\sequence{\mcf}[n][\nset]$). The advantage is to avoid writing all
  likelihoods as functions of the observations, yielding more compact
  expressions and arguments. In order to illustrate better how this setting can be used 
  in practice, we describe it in the simple i.i.d. case.

  Consider the dominated i.i.d. case comprising an $n$-sample $\chunk{Y}{1}{n}=(Y_1,\dots,Y_n)$ of i.i.d. observations
  taking on values in $(\Yset,\Ysigma)$ and having marginal density $\Td[\thv]$, depending on
  some unknown parameter $\thv\in\Theta$, w.r.t. a given $\sigma$-finite
  dominating measure $\nu$ on $(\Yset,\Ysigma)$. In this case, we set
  $\sequence{Y}[n][\nsetpos]$ as the canonical process defined on
  $\Omega=\Yset^{\nsetpos}$ endowed with $\mcf=\Ysigma^{\tensprod\nsetpos}$, the
  $\sigma$-field generated by cylinder sets. Then
  $\mcf_n=\sigma(\chunk{Y}{1}{n})$ and $\gd{\theta}{n}=\Td[\theta]^{\tensprod
    n}(\chunk{Y}{1}{n})$, where $\Td[\theta]^{\tensprod n}$ denotes the $n$th
  self tensor product of $\Td[\theta]$ defined on $\Yset^n$.  While
  $\Td[\theta]^{\tensprod n}$ is a density w.r.t. the product measure
  $\nu^{\tensprod n}$, $\gd{\theta}{n}$ is a density with respect to the
  $\sigma$-finite measure $\nu_n$ given by 
\begin{equation}
  \label{eq:nu-n-def-canonic-1}
\nu_n(B)=\nu^{\tensprod n}(A)\eqsp,\ B\in\mcf_n
\text{ with $B={[\chunk Y1n]}^{-1}(A)$ and $A\in\Ysigma^{\tensprod n}$}\eqsp.
\end{equation}
In the i.i.d. case, one assumes that there is a true parameter
$\thv\in\Theta$ and $\PP$ is the corresponding distribution of
$\sequence{Y}[n][\nsetpos]$, implying that  $\gdstar{n}=\gd{\thv}{n}$.

Note that the $\mcf_n\to\Ysigma^{\tensprod n}$-mapping defined by $B\mapsto
\chunk Y1n(B)=\{\chunk Y1n(\omega), \omega\in B\}$ is bijective and can be
seen as the (unique) reciprocal of the preimage mapping $[\chunk
Y1n]^{-1}:\Ysigma^{\tensprod n}\to\mcf_n$.  Hence $\nu_n$
in~\eqref{eq:nu-n-def-canonic-1} can be equivalently defined as
\begin{equation}
  \label{eq:nu-n-def-canonic}
  \nu_n=\nu^{\tensprod n}\circ\chunk Y1n\text{ on $\mcf_n$}\Longleftrightarrow
  \nu_n\circ[\chunk Y1n]^{-1}=\nu^{\tensprod n}\text{ on $\Ysigma^{\tensprod n}$}\eqsp.
\end{equation}

\section{Postponed proof of general results}
\label{sec:proof-crefthm:general}
\subsection{Proof of \Cref{prop:cns:remote}}
\label{subsec:proof:prop:cns:remote}
Let $A \in \mathcal{T}$ and suppose that $B_n \in \mcf_n$ for all $n\in\nset$ and that \eqref{eq:cns:theta}
and \eqref{eq:cns:thv} hold. Then there exists $\rho>1$ such that
\begin{equation*}
 \PE\lr{\sum_{n=1}^\infty \rho^n \int_A \prior(\rmd \theta) \, \frac{\gd{\theta}{n}}{\gdstar{n}} \oneSub{B_n}}
 =\sum_{n=1}^\infty \rho^n \int_A \prior(\rmd \theta) \, \PP_{\theta,n} (B_n)<\infty\eqsp.
\end{equation*}
This implies
$$
\sum_{n=1}^\infty \rho^n \int_A \prior(\rmd \theta) \, \frac{\gd{\theta}{n}}{\gdstar{n}} \oneSub{B_n} <\infty\eqsp, \quad \PP\as
$$
Now, the set $\Omega_0 = \liminf_{n \to \infty} B_n \in\mcf$ of $\PP$-probability one in
\eqref{eq:cns:thv} is such that for all $\omega \in \Omega_0$,
$\{n \in \nset : \oneSub{B_n}(\omega)=0\}$ is finite. Thus, the series $\sum_{n=1}^\infty
\rho^n {\int_A \prior(\rmd \theta) \, \gd{\theta}{n}}/{\gdstar{n}} $ is convergent
$\PP\as$, establishing that $A$ is $\PP$-remote.

Conversely, if  $A \in \mathcal{T}$ is $\PP$-remote, then choose $\rho>1$ such that
\begin{equation} \label{eq:remote:cond}
\limsup_{n \to \infty} n^{-1} \log \left( \int_A \prior(\rmd \theta) \, \frac{\gd{\theta}{n}}{\gdstar{n}} \right) \leq -\log \rho<0 \quad \PP\as
\end{equation}
Pick $ \tilde \rho \in (1,\rho)$ and $\varrho \in (1/\tilde \rho,1)$. Set
$$
B_n\eqdef\lrc{\int_A \prior(\rmd \theta) \, \frac{\gd{\theta}{n}}{\gdstar{n}} \leq \varrho^n}\eqsp.
$$
Then $B_n\in\mcf_n$ and, by Tonelli's theorem,
$$
\int_A \prior(\rmd \theta) \, \PP_{\theta,n}(B_n) \leq \varrho^n \PP(B_n) \leq \varrho^n \quad \PP\as;
$$
thus, \eqref{eq:cns:theta} is satisfied. Since $\tilde \rho<\rho$, \eqref{eq:remote:cond} implies
$$
\sup_{n \in \nset} \tilde \rho^n \int_A \prior(\rmd \theta) \, \frac{\gd{\theta}{n}}{\gdstar{n}} <\infty  \quad \PP\as,
$$
so that
$$
 \PP(\ensemble{\omega \in \Omega}{\omega\in B_n^c \ \; \mathrm{i.o.}}) = \PP\lr{ \tilde \rho^n \int_A \prior(\rmd \theta) \, \frac{\gd{\theta}{n}}{\gdstar{n}}  > (\tilde \rho\varrho)^n \ \; \mbox{i.o.}}=0\eqsp,
 $$
and \eqref{eq:cns:thv} holds.

\section{Useful lemmas}
\begin{lemma}
  \label{lem:zarbi}
Let $Z$ be a $[0,\infty]$-valued variable on $(\Omega,\mcf,\PP)$ and let $\mcg$ be a
sub-$\sigma$-field of $\mcf$. If $\cesp{Z}{\mcg}<\infty$ $\PP\as$, then
$Z<\infty$  $\PP\as$
\end{lemma}

\begin{proof}
  Let  $B=\{Z=\infty\}$. Then we have
$$
\cesp{Z}{\mcg}\geq\cesp{Z\1B}{\mcg}=\infty\ \text{on}\;\left\{\PP(B\,|\,\mcg)>0\right\}.
$$
Hence, if  $\cesp{Z}{\mcg}<\infty$ $\PP\as$, then $\PP(B\,|\,\mcg)=0$ $\PP\as$
and so $\PP(B)=0$.
\end{proof}

\begin{lemma} \label{lem:ergodic:indep:product}
  Let $Q$ and $Q'$ be two Markov kernels on $(\Xset,\Xsigma)$ and
  $(\Xset',\Xsigma')$, respectively. Let $\bar Q$ the Markov kernel on
  $(\Xset\times\Xset',\Xsigma\tensprod\Xsigma')$ defined by, for all $A \in \Xsigma$ and $B \in \Xsigma'$, 
$$
\bar Q((x,x'),A\times B)=Q(x,A) Q'(x,B) \eqsp.
$$
Suppose that $Q$ and $Q'$ are ergodic with stationary distributions $\pi$ and
$\pi'$, respectively, in the sense that for all initial distributions $\Xinit$ and $\Xinit'$ on $(\Xset,\Xsigma)$ and
$(\Xset',\Xsigma')$, respectively, it holds that
$$
\lim_{n \to \infty}\|\Xinit Q^n -\pi\|_{\mathrm{TV}}=0
\text{ and }
\lim_{n \to \infty}\|\Xinit' Q^{\prime n} -\pi'\|_{\mathrm{TV}}=0
\eqsp.
$$
Then $\bar Q$ is ergodic with stationary distribution
$\pi\tensprod\pi'$.
\end{lemma}

\begin{proof}
  Let $h$ be a measurable function on
  $\Xset\times\Xset'$ such that $| h | \leq 1$. For all $n\in \nsetpos$, we
  may write, for $(x,x')\in\Xset\times\Xset'$,  $\bar Q^n h(x,x')-(\pi\tensprod\pi')h$ as
   \begin{multline*}
\int \left(\int h(x_n,x'_n)\, Q^n(x,\rmd x_n)-\int  h(x_n,x'_n) \, \pi(\rmd x_n)\right)
Q^{\prime n}(x',\rmd x'_n)\\
+\int \left(\int  h(x_n,x'_n) \, Q^{\prime n}(x',\rmd x'_n)
-\int  h(x_n,x'_n)\,\pi'(\rmd x'_n)\right) \, \pi(\rmd x_n)\eqsp,
  \end{multline*}
from which we immediately deduce that
$$
\lim_{n \to \infty}\|\delta_{(x,x')}Q^{\prime n} -\pi\tensprod\pi'\|_{\mathrm{TV}}=0\eqsp.
$$
The ergodicity of $\bar Q$ follows.
\end{proof}
\bibliographystyle{plain}
\bibliography{dorbibs}

\end{document}